\pdfoutput=1
\documentclass[10pt, letterpaper, reqno]{amsart}

\usepackage[utf8]{inputenc}
\usepackage[dvipsnames]{xcolor}
\definecolor{darkpurple}{HTML}{583580}
\definecolor{deepblue}{HTML}{174A7E}
\usepackage[left=2.4cm,right=2.4cm,top=2.5cm,bottom=2.5cm]{geometry}
\usepackage{amsmath,amsfonts, amssymb, stmaryrd}
\usepackage{enumerate}
\usepackage{mathtools}
\usepackage{thmtools}
\usepackage{tikz-cd}
\usepackage{mathrsfs}
\usepackage[english]{babel}
\usepackage[pagebackref, colorlinks, linkcolor=deepblue, citecolor=deepblue, urlcolor=darkpurple, hypertexnames=false]{hyperref}
\usepackage[capitalize]{cleveref}

\numberwithin{equation}{section}

\theoremstyle{plain}
\newtheorem{thm}{Theorem}[subsection]
\newtheorem{prop}[thm]{Proposition}

\newtheorem{cor}[thm]{Corollary}
\newtheorem{thmintro}{Theorem}

\theoremstyle{definition}
\newtheorem{defi}[thm]{Definition}

\newtheorem{notation}[thm]{Notation}

\declaretheorem[name=Remark,sibling=thm,qed={\lower-0.3ex\hbox{$\diamond$}}]{rmk}
\declaretheorem[name=Remark,unnumbered,qed={\lower-0.3ex\hbox{$\diamond$}}]{remarknonumber}

\makeatletter
\newcommand{\colim@}[2]{%
 \vtop{\m@th\ialign{##\cr
   \hfil$#1\operator@font colim$\hfil\cr
   \noalign{\nointerlineskip\kern1.5\ex@}#2\cr
   \noalign{\nointerlineskip\kern-\ex@}\cr}}%
}
\newcommand{\colim}{%
 \mathop{\mathpalette\colim@{\rightarrowfill@\scriptscriptstyle}}\nmlimits@
}
\renewcommand{\varprojlim}{%
 \mathop{\mathpalette\varlim@{\leftarrowfill@\scriptscriptstyle}}\nmlimits@
}
\renewcommand{\varinjlim}{%
 \mathop{\mathpalette\varlim@{\rightarrowfill@\scriptscriptstyle}}\nmlimits@
}
\makeatother

\newenvironment{smatrix}{\left( \begin{smallmatrix} } {\end{smallmatrix} \right) }
\newcommand{\stbt}[4]{\begin{smatrix}#1 & #2 \\ #3 & #4\end{smatrix}}

\renewcommand{\le}{\leqslant}
\renewcommand{\geq}{\ge}
\renewcommand{\ge}{\geqslant}

\newcommand{\Aa}{\mathcal{A}}
\newcommand{\Ee}{\mathcal{E}}
\newcommand{\hh}{\mathcal{H}}
\newcommand{\Mm}{\mathcal{M}}
\newcommand{\Nn}{\mathcal{N}}
\newcommand{\Oo}{\mathcal{O}}
\newcommand{\Ss}{\mathcal{S}}
\newcommand{\Uu}{\mathcal{U}}
\newcommand{\Ww}{\mathcal{W}}

\newcommand{\M}{\mathfrak{M}}
\newcommand{\Z}{\mathbb{Z}}
\newcommand{\Q}{\mathbb{Q}}
\newcommand{\Qp}{\Q_p}
\newcommand{\Zp}{\Z_p}
\newcommand{\Ff}{\mathbb{F}}
\newcommand{\Fp}{\Ff_{p}}
\newcommand{\R}{\mathbb{R}}
\newcommand{\C}{\mathbb{C}}
\newcommand{\A}{\mathbb{A}}
\newcommand{\T}{\mathbb{T}}
\newcommand{\p}{\mathfrak{p}}
\newcommand{\cc}{\mathfrak{c}}
\newcommand{\X}{\mathfrak{X}}
\newcommand{\kk}{\underline{k}}
\newcommand{\f}{\mathrm{f}}
\newcommand{\Af}{\A_\f}
\newcommand{\AFf}{\A_{F, \f}}
\newcommand{\Ig}{\mathfrak{Ig}}

\newcommand{\sym}{\operatorname{Sym}}
\newcommand{\imp}{\mathrm{imp}}
\newcommand{\new}{\mathrm{new}}
\newcommand{\res}{\operatorname{Res}}

\newcommand{\pr}{\operatorname{pr}}

\newcommand{\Pif}{\Pi_\f}

\newcommand{\m}{\mathrm{m}}
\newcommand{\et}{\text{\textup{\'et}}}
\newcommand{\naive}{\text{\textup{na\"\i ve}}}

\DeclareMathOperator{\ord}{ord}
\DeclareMathOperator{\ch}{ch}
\DeclareMathOperator{\As}{As}
\DeclareMathOperator{\GL}{GL}
\DeclareMathOperator{\RG}{R\Gamma}
\DeclareMathOperator{\tor}{tor}

\DeclareMathOperator{\dr}{dR}
\DeclareMathOperator{\tr}{Tr}
\DeclareMathOperator{\sgn}{sgn}
\DeclareMathOperator{\re}{Re}

\DeclareMathOperator{\nm}{Nm}
\DeclareMathOperator{\myom}{\underline{\omega}}

\newcommand{\sph}{\mathrm{sph}} % not an operator
\newcommand{\myomk}{\myom^{(\kk, w)}}
\makeatletter
\@namedef{subjclassname@2020}{%
 \textup{2020} Mathematics Subject Classification}
\makeatother

\begin{document}

 \title[$p$-adic Asai L-functions]{$p$-adic Asai $L$-functions for quadratic Hilbert eigenforms}
 %\date{\today}

 \author{Giada Grossi}
 \address[G.~Grossi]{CNRS, DMA--\'Ecole Normale Supérieure, Paris, France}
 \email{giada.grossi@ens.fr}
 \urladdr{\url{https://orcid.org/0000-0002-0883-4494}}

 \author{David Loeffler}
 \address[D.~Loeffler]{UniDistance Suisse, Schinerstrasse 18, 3900 Brig, Switzerland}
 \email{david.loeffler@unidistance.ch}
 \urladdr{\url{https://orcid.org/0000-0001-9069-1877}}
 \thanks{Supported by European Research Council Consolidator Grant No. 101001051 (\textsc{ShimBSD}) (Loeffler) and US National Science Foundation Grant No. DMS-1928930 (all authors)}

 \author{Sarah Livia Zerbes}
 \address[S.L.~Zerbes]{Department of Mathematics, ETH Z\"urich, R\"amistrasse 101, 8092 Z\"urich, Switzerland}
 \email{sarah.zerbes@math.ethz.ch}
 \urladdr{\url{https://orcid.org/0000-0001-8650-9622}}

 \begin{abstract}
  We construct $p$-adic Asai $L$-functions for cuspidal automorphic representations $\Pi$ of ${\rm Res}_{F/\Q}{\rm GL}_2$, where $F$ is a real quadratic field in which $p$ splits. Our method relies on higher Hida theory for Hilbert modular surfaces with Iwahori level at one prime above $p$.
 \end{abstract}

 \maketitle
 \tableofcontents
 \addtocontents{toc}{\protect\setcounter{tocdepth}{1}}

\section{Introduction}

 The study of critical values of $L$-functions is central in many areas of modern number theory. The Bloch--Kato and Beilinson conjectures relate critical values of complex $L$-functions to certain arithmetic data; and much progress towards such conjectures has been made by $p$-adic methods, in which a fundamental role is played by $p$-adic $L$-functions, interpolating $p$-adically the algebraic parts of critical values of a given complex $L$-function or family of $L$-functions. In this paper, we construct such a $p$-adic counterpart for the \emph{Asai}, or \emph{twisted tensor}, $L$-function attached to a Hilbert modular form over a real quadratic field $F$.

 \subsection{Higher Hida theory for $\p_1$-ordinary cohomology}

  The main tool we shall use is a version of \emph{higher Hida theory} for Hilbert modular forms at primes $p = \p_1 \p_2$ split in $F$. Such a theory was initially introduced in the first author's paper \cite{myhida}; in this paper we develop a variant assuming only ordinarity at a single prime above $p$, and we study functoriality properties of this theory for the embedding $\GL_2 \hookrightarrow \res_{F/\Q} \GL_2$, defining pullback and pushforward maps between $p$-adic cohomology spaces for these two groups.

  For our first result, write $X_{G, 0}(\p_1)_{\Qp}$ for the compactified Hilbert modular surface over $\Qp$ with Iwahori level at $\p_1$ and spherical level at $\p_2$ (and some fixed prime-to-$p$ level structure $K_G^{(p)}$). Denote by $\sigma_1,\sigma_2$ the two real embeddings of $F$; fixing an isomorphism $\overline{\Q}_p \cong \C$,  we may assume $\sigma_i$ corresponds to the prime $\p_i$. 
  We also fix $(k_1, k_2, w) \in \Z^3$ with $k_1 = k_2 = w \bmod 2$ and for any integer $a$ we denote by $\myom^{(k_1+2a,k_2,w)}$ the automorphic sheaf on $X_{G, 0}(\p_1)_{\Qp}$ of weight $(k_1+2a,k_2,w)$ (see \eqref{eq:defiomegaGmodel}).

  \begin{thmintro}
   \label{thmintro:classicity}
   There exists a perfect complex $M_G^\bullet(k_1 + 2\kappa, k_2, w)$ of modules over $\Lambda = \Zp[[\Zp^\times]]$, with an action of prime-to-$p$ Hecke operators, such that for all integers $a \in \Z$ with $k_1 + 2a \ge 3$, we have a Hecke-equivariant isomorphism
   \[
    M_G^{\bullet}(k_1 + 2\kappa, k_2, w)\otimes^{\mathbb{L}}_{\Lambda,a} \Qp \simeq e({U}_{\p_1})\RG\left(X_{G, 0}(\p_1)_{\Qp}, \myom^{(k_1+2a,k_2,w)}\right),
   \]
   where $\kappa$ denotes the universal character of $\Lambda$.
   There is also an analogous complex $M_G^\bullet(k_1 + 2\kappa, k_2, w; -D)$ interpolating the cohomology of the cuspidal sheaves $\myom^{(k_1+2a,k_2,w)}(-D)$.
  \end{thmintro}

  See \cref{thm:classicalityigusa} below. Here $e({U}_{\p_1})$ is the ordinary projector associated to the Hecke operator at $\p_1$; we do not impose any ordinarity conditions at $\p_2$. 

  \begin{remarknonumber}
   Observe that the perfectness assertion in Theorem \ref{thmintro:classicity} does not imply that all ordinary eigensystems can be interpolated in one-parameter families (with $k_1$ varying and $k_2$ fixed), since we are not claiming that the complex is concentrated in a single degree; so its cohomology groups may not be flat over $\Lambda$ and their formation may not commute with specialization. (We do show that the cuspidal complex $M_G^\bullet(k_1 + 2\kappa, k_2, w; -D)$ is concentrated in a single degree under an additional assumption on $k_2$, see \cref{rmk:onedegree}.)
   In \cref{sect:eigenfamilies} we prove a slightly weaker result, showing that eigensystems associated to classical ordinary cusp forms of regular weight do vary in families over a sufficiently small open affinoid in weight space. 
  \end{remarknonumber}

  For our second result, let $M_H(\tau)$ denote the space of $\Lambda$-adic elliptic modular forms of weight $\tau :\Zp^\times \to \Lambda^\times$ and tame level $K_H^{(p)} = K_G^{(p)} \cap H$ (see \cref{sect:funcIgusa} below).

  \begin{thmintro}
   There exists a pushforward map of $\Lambda$-modules
   \[ \iota_\star : M_H(k_1 - k_2 + 2\kappa) \to H^1(M_G^\bullet(k_1 + 2\kappa, 2 - k_2, w)) \]
   which interpolates, via the isomorphism of \cref{thmintro:classicity}, the projection to the $U_{\p_1}$-ordinary parts of the pushforward maps in coherent cohomology in all weights $(k_1 + 2a, 2-k_2, w)$ with $k_1 + 2a \ge 3$.
  \end{thmintro}

  Note that the integer-weight specialisations of forms in the space $M_H(\tau)$ can be classical holomorphic modular forms of Iwahori level at $p$, but may also include holomorphic forms of higher $p$-power levels, or nearly-holomorphic modular forms (via the $p$-adic unit root splitting); and we show that the specialisations of our $p$-adic pushforward map are compatible with algebraic pushforward maps on these larger spaces. This is crucial for our applications to $p$-adic $L$-functions, since the $\GL_2$ Eisenstein series we consider are only nearly-holomorphic in general.

  \begin{remarknonumber}
   In this paper we shall apply Theorem B to pushforwards of Eisenstein families, but the case of pushforwards of cuspidal families is also interesting, since these pushforwards are related to the square roots of twisted triple-product $L$-functions. We hope to investigate this case in future work.
  \end{remarknonumber}

 \subsection{One-variable $p$-adic $L$-functions}

  Our first application of the above theory is to construct a one-variable $p$-adic Asai $L$-function attached to a single cuspidal automorphic representation $\Pi$ of $G \coloneqq \res_{F/\Q}\GL_2$, interpolating critical values of the Asai $L$-function of $\Pi$ and its twists by $p$-power-conductor Dirichlet characters.

  Our conventions are such that $\Pi$ is unitary, and its central character $\chi_\Pi$ is a finite-order Gr\"ossencharacter of $F$. The weight of $\Pi$ is given by a pair of positive integers $(k_1, k_2)$; we shall only consider ``paritious'' weights (i.e.~with $k_1 = k_2 \bmod 2$), and in order to have critical values for the Asai $L$-function, we need to assume $k_1 \ne k_2$. Without loss of generality we suppose $k_1 > k_2 \ge 1$. We suppose $p$ is a prime split in $F$ and coprime to the level of $\Pi$.%

  We write $L_{\As}(\Pi, s)$ for the Asai $L$-function (normalised in the analytic fashion, so the centre of the functional equation is at $s = \tfrac{1}{2}$), and $L_{\As}(\Pi, \chi, s)$ for the $L$-function twisted by a Dirichlet character $\chi$. Here we define the local $L$-factors at bad primes via the local Langlands correspondence, or equivalently as the common denominator of the Asai zeta integral defined in \S \ref{sect:zetadef} (for the equivalence of these definitions see \cite[Remark 2.14]{loefflerwilliams18}). The critical points of $L_{\As}(\Pi, \chi, s)$, for any $\chi$, are given by
  \begin{equation}\label{eq:critrangeintro}
   \tag{crit}
   s \in \Z, \qquad \tfrac{2-k_1+k_2}{2} \le s \le \tfrac{k_1-k_2}{2}.
  \end{equation}

  The general conjectures of Coates--Perrin-Riou \cite{coatesperrin} and Panchishkin \cite{pan94} on the expected form of $p$-adic $L$-functions associated to motives over $\Q$, in this case, predict the following: if a suitable ``ordinarity'' condition (the Panchishkin condition) holds, the critical $L$-values of $L_{\As}(\Pi, \chi, s)$, for Dirichlet characters $\chi$ of $p$-power conductor, can be interpolated by a $p$-adic measure after multiplying by a certain Euler factor $\Ee_p(\As(\Pi), \chi, s)$. See \cite[Definition 5.5]{pan94} for the definition of Panchishkin's condition; in our setting (with $k_1 > k_2$) the Panchishkin condition amounts to ordinarity of $\Pi$ at $\p_1$, while the stronger ordinarity condition assumed in \cite{coatesperrin} amounts to ordinarity at both $\p_1$ and $\p_2$.

  \begin{notation}\label{not:alphabeta}
   For $i = 1, 2$, we write $\alpha_i^\circ, \beta_i^\circ$ for the Satake parameters of $\Pi_{\p_i}$ (unitarily normalised, so that $|\alpha_i^\circ| = |\beta_i^\circ| = 1$), and define
   \[ \alpha_i = p^{(k_i - 1) / 2} \alpha_i^\circ, \quad \beta_i = p^{(k_i - 1) / 2} \beta_i^\circ. \]
  \end{notation}

  We will see below that $\alpha_1, \beta_1$ are $p$-adically integral (for our chosen embedding into $\overline{\Q}_p$), and the ordinarity condition at $\p_1$ is equivalent to precisely one of $\alpha_1$ and $\beta_1$ being a $p$-adic unit, which we may take to be $\alpha_1$.

  \begin{defi}
   Suppose $\Pi$ is ordinary at $\p_1$. For $s \in \Z$ and $\nu$ a Dirichlet character of $p$-power conductor, define
   \[
    \Ee_p(\As(\Pi), \nu, s) =
    \begin{cases}
     \left(1 - \frac{p^{s-1}}{\alpha_1^\circ \alpha_2^\circ}\right)
     \left(1 - \frac{p^{s-1}}{\alpha_1^\circ \beta_2^\circ}\right)
     \left(1 - \frac{\beta_1^\circ \alpha_2^\circ}{p^s}\right)
     \left(1 - \frac{\beta_1^\circ \beta_2^\circ}{p^s}\right) &
     \text{if $\nu = 1$}, \\[2mm]
     \tfrac{1}{G(\nu)^2} \cdot \left( \tfrac{p^{2s}}{(\alpha_1^\circ)^2\alpha_2^\circ\beta_2^\circ}\right)^r & \text{if $\nu$ has conductor $p^r$, $r \ge 1$}.
    \end{cases}
   \]
  \end{defi}
  where $G(\nu) = \sum_{a \in (\Z / p^r)^\times} \nu(a) e^{2\pi i a / p^r}$ is the Gauss sum. This is exactly the modified Euler factor predicted by the conjectures of \cite{coatesperrin, pan94}.

  \begin{thmintro}\label{thmintro:cycloL}
   Suppose $\Pi$ is ordinary at $\p_1$ and its level is coprime to $p$. Then there exists a $p$-adic measure $\mathcal{L}_{p, \As}(\Pi) \in \Lambda_E =\Zp[[\Zp^{\times} ]]\otimes E$, for $E / \Qp$ a sufficiently large finite extension, and a period $\Omega(\Pi) \in \C^\times$, such that for every $s$ satisfying \eqref{eq:critrangeintro} and Dirichlet character $\nu$ of conductor $p^r$, for $r \ge 0$, we have
   \[
    \int_{\Zp^\times}\nu^{-1}(x)x^s\, \mathrm{d}\mathcal{L}_{p, \As}(\Pi) = \Ee_p(\As(\Pi), \nu, s) \cdot \frac{\Gamma(s + \tfrac{k_1+k_2-2}{2}) \Gamma(s + \tfrac{k_1 - k_2}{2})}{2^{(k_1-2)} i^{(1-k_2)} (-2\pi i)^{(2s + k_1 - 1)}} \cdot \frac{L_{\As}\left(\Pi, \nu, s\right)}{\Omega(\Pi)}.
   \]
  \end{thmintro}

  Note that both sides of the equality of \cref{thmintro:cycloL} are algebraic (and lie in $L(\nu)$ for a number field $L$ depending only on $\Pi$), which is an extension of existing algebraicity results due to Shimura \cite{shimura-hilbertII}.

 \subsection{Two-variable $p$-adic $L$-functions}

  Our second main result involves a slight modification of the Asai $L$-function, denoted $L_{\As}^{\imp}(\Pi, \chi, s)$, which we call the \emph{imprimitive} Asai $L$-function. Unlike the primitive Asai $L$-function, it is given by a straightforward formula in terms of Hecke eigenvalues: if $a_{\mathfrak{m}}^\circ(\Pi)$ is the normalised Hecke eigenvalue of the new vector of $\Pi$ at an ideal $\mathfrak{m} \trianglelefteqslant \Oo_F$, then the imprimitive Asai $L$-function is given by $(\star) \cdot \sum_{n \in \mathbb{N}} a^\circ_{n \Oo_F}(\Pi) n^{-s}$ where $(\star)$ is a Dirichlet $L$-function. (See \S\ref{sec:imprimitiveasai} for a more precise, representation-theoretic definition). This differs from $L_{\As}(\Pi, \chi, s)$ by a finite product of Euler factors, namely
  \[
   L_{\As}^{\imp}(\Pi,\chi,s)=L_{\As}(\Pi, \chi, s)\cdot \prod_{\ell \mid {\rm Nm}(\mathfrak{N})}C_{\ell}(\ell^{-s}\chi(\ell)),
  \]
  where $\mathfrak{N} \trianglelefteqslant \Oo_F$ is the conductor of $\Pi$, and $C_{\ell}(X)$ is a polynomial dividing the Euler factor of $L_{\As}(\Pi,\chi,s)$ at $\ell$. Since we can clearly define measures interpolating the factors $C_\ell(-)$, it follows \emph{a fortiori} from \cref{thmintro:cycloL} that there exists a measure $\mathcal{L}_{p, \As}^{\imp}(\Pi)$ interpolating critical values of $L^{\imp}_{\As}(\Pi, \chi, s)$. We shall show that these ``imprimitive $p$-adic $L$-functions'' vary analytically when $\Pi$ itself is allowed to vary in a family.

  \begin{thmintro}
   \label{thmintro:twovar}
   Let $\Pi$ be ordinary at $\p_1$, with level coprime to $p$, and with weight $(k_1, k_2)$ such that $k_1 = k_2 \bmod 2$ and $k_1 \ge 3, k_2 \ge 2$. (We no longer assume that $k_1 > k_2$).

   \begin{enumerate}
    \item There exists a $p$-adic family $\underline{\Pi}$ of automorphic representations parametrised by some affinoid neighbourhood $\Uu$ of 0 in the $p$-adic weight space $\Ww$, with weight $(k_1 + 2\kappa_\Uu, k_2)$ where $\kappa_{\Uu}$ is the universal character of $\Uu$, whose specialisation at $0 \in \Uu$ is $\Pi$.

    \item There exists an element $\mathcal{L}_{p, \As}^{\imp}(\underline{\Pi}) \in \Oo(\Uu) \mathop{\hat\otimes} \Lambda_E$, whose specialisation at any integer $a \in \Uu \cap \Z$ with $k_1 + 2a > k_2$ is a non-zero multiple of $\mathcal{L}_{p, \As}^{\imp}(\Pi[a])$, where $\Pi[a]$ is the specialisation of $\underline{\Pi}$ in weight $(k_1 + 2a, k_2)$.
   \end{enumerate}
  \end{thmintro}

  In particular, this allows us to define the one-variable $p$-adic $L$-function $\mathcal{L}_{p, \As}^{\imp}(\Pi)$ without assuming $k_1 > k_2$, since the interpolating property of the two-variable measure uniquely determines the fibre at $a = 0$, even when the Asai $L$-function has no critical values. This includes the important case of base-changes of elliptic modular forms.

  \begin{rmk}
   It should be possible to interpolate the ``primitive'' $p$-adic $L$-functions $\mathcal{L}_{p,\As}(\Pi[a])$ as $a$ varies, not just their imprimitive versions, but this would require a careful study of local test vectors at primes $\ell \ne p$ in $p$-adic families (analogous to the analysis carried out in \cite{chenhsieh} for Rankin--Selberg $L$-functions), and we shall not undertake this here for reasons of space.
  \end{rmk}

 \subsection{Relation to other work}

  The results described above are in many ways parallel to the results of \cite{padicLfnct} in the case of spin $L$-functions of genus 2 Siegel modular forms. In particular, both here and in \emph{op.cit.}, the underlying $L$-value formula requires us to work with coherent cohomology classes in intermediate degrees (not corresponding to either holomorphic or anti-holomorphic forms) and hence the use of higher Hida theory is essential. This contrasts with earlier works in which the theory of ordinary $p$-adic families of modular forms has been used to construct $p$-adic $L$-functions, such as for $\GL_2 \times \GL_2$ \cite{hidapadic,panchishkin82}, and the $\GL_2$ triple product \cite{harristilouine}. Classical Hida theory is sufficient for those cases since one works with products of the modular curve and it suffices to vary $p$-adically the degree zero cohomology group. Another case where higher Hida theory plays a role in constructing $p$-adic $L$-functions is $U(2, 1)$, which has been treated by Oh \cite{oh}.

  In a companion paper \cite{asairegulator}, we relate the values of the $p$-adic Asai $L$-functions constructed here, at points \emph{outside} the range of interpolation, to the Euler system classes in Galois cohomology constructed in \cite{HMS}. This will be used in subsequent work to prove an explicit reciprocity law for these Euler system classes, and hence obtain new cases of the Bloch--Kato conjecture and Iwasawa main conjecture for Asai motives. See \cite{LZ-improved} for a study of exceptional-zero phenomena for our $p$-adic Asai $L$-functions, and \cite{kazi-loeffler} for a generalisation to finite-slope settings using higher Coleman theory.

  The main result of \cite{loefflerwilliams18} is a counterpart of \cref{thmintro:cycloL} for automorphic representations of $\GL_2$ over \emph{imaginary} quadratic fields. However, the two constructions use very different techniques; indeed, the construction of the $p$-adic $L$-function in the imaginary case is more strongly related to the construction of the Euler system of \cite{HMS} in the real-quadratic case, rather than that of the $p$-adic $L$-function carried out here.

  %%%%%%%%%%%%%%%%%%%%%%%%%%%%%%%%%%%%%%%%%%%%%%%%%%%%%
\section{Higher Hida theory for \texorpdfstring{$\p_1$}{p1}-ordinary cohomology}
 %%%%%%%%%%%%%%%%%%%%%%%%%%%%%%%%%%%%%%%%%%%%%%%%%%%%%

 We adapt the constructions of \cite{myhida} to a ``partially ordinary'' setting, only imposing an ordinarity restriction at one of the primes above $p$ (and correspondingly only varying one component of the weight). Assume that the prime $p$ splits in $F$. Let us write $p=\p_1\cdot \p_2$, and let $\sigma_i$ be the complex embedding corresponding to $\p_i$ for our fixed isomorphism $\overline{\Q}_p \cong \C$.

 \begin{rmk}
  More generally, we could choose an arbitrary totally real field $F$ with $p$ completely split, and choose an arbitrary subset of primes above $p$; we stick to the case $[F : \Q] = 2$ for simplicity, and leave the generalisation to the interested reader.
 \end{rmk}

 \subsection{Preliminaries}

  \subsubsection{Hilbert modular surfaces} 

   For $n \ge 0$, we consider the level group $K_n \subset \GL_2(\AFf)$ defined by 
   \[ K_n = K^{(p)} \cdot K_{\p_1,n} \cdot \GL_2(\Oo_{F, \p_2}), \]
   where $K_{\p_1,n}=\{g\in \GL_2(\Oo_{F, \p_1}): g\equiv \stbt{*}{*}{0}{*} \mod p^n\}$ and $K^{(p)}$ is a neat open compact subgroup of $G(\Af^{(p)})$. 
   We let $Y_0(\p_1^n)_{\Q}$ denote the canonical $\Q$-model of the Hilbert modular variety of level $K_n$, and $X_0(\p_1^n)_{\Q}$ a toroidal compactification of $Y_0(\p_1^n)_{\Q}$. The toroidal compactification depends on a choice of polyhedral cone decomposition; we can, and do, choose this so that $X_0(\p_1^n)_{\Q}$ is smooth and projective, and its boundary $D$ is a normal-crossing divisor. For $n = 0$ we write simply $Y_{\sph, \Q}$ and $X_{\sph,\Q}$.

  \subsubsection{Moduli spaces}

   Since the Hilbert modular varieties are not of PEL type, we shall also consider an auxiliary moduli space, as in \cite[\S 2.1]{myhida} (see also \cite{tianxiao16}). For an ideal $\cc \trianglelefteqslant \Oo_F$ coprime to $p$ and the level of $K^{(p)}$, we let $\Mm_0^{\cc}(\p_1^n)_{\Q}$ denote the moduli space of tuples $(A,\iota, \lambda, \alpha_{K^{(p)}}, H)$, where $(A, \iota, \lambda, \alpha_{K^{(p)}})$ is an abelian surface with $\Oo_F$-action, $\cc$-polarization and level $K^{(p)}$-structure as in \emph{op.cit.}, and $H$ is a cyclic subgroup of $A[\p_1^n]$ of order $p^n$. We fix a set of ideals representing the narrow class group of $F$, and write $\Mm_0(\p_1^n)_{\Q} = \bigsqcup_{[\cc] \in \operatorname{Cl}^+_F} \Mm_0^{\cc}(\p_1^n)_{\Q}$. 

   Each of the varieties $\Mm_0^{\cc}(\p_1^n)_{\Q}$ has an action of the totally-positive units $\Oo_{F,+}^\times$ (acting by scaling the polarization $\lambda$); this factors through the finite group
   \[ 
    \Delta_{K^{(p)}} \coloneqq \Oo_{F, +}^\times / \left\{ u^2 : \stbt{u}{}{}{u} \in K^{(p)} \right\}, 
   \]
   and we have an isomorphism
   \begin{equation}
    \label{eq:modspacetoShvar}
    Y_0(\p_1^n)_{\Q} \cong \Mm_0(\p_1^n)_{\Q} / \Delta_{K^{(p)}} =  \bigsqcup_{[\cc]\in \operatorname{Cl}_{F}^{+}} \Mm_0^{\cc}(\p_1^n)_{\Q} / \Delta_{K^{(p)}}.
   \end{equation}
   We write $\Mm_0^{\cc}(\p_1^n)_{\Q}^{\tor}$, $\Mm_0(\p_1^n)_{\Q}^{\tor}$ for a choice of toroidal compactification, which we may choose so that the isomorphism \eqref{eq:modspacetoShvar} extends to an isomorphism
   \[ X_0(\p_1^n)_{\Q} \cong \Mm_0(\p_1^n)^{\tor}_{\Q} / \Delta_{K^{(p)}}.\]

   Again, for $n = 0$ we use the alternative notations $\Mm^{\cc}_{\sph, \Q}$, $\Mm_{\sph, \Q}$, and $\Mm_{\sph, \Q}^{\tor}$.
  \subsubsection{Integral models}

   For $n =0, 1$, which are the main cases of interest for us, the varieties defined above have canonical $\Z_{(p)}$-models (see \cite{pappas}). As we are mostly interested in $p$-adic theory, we shall write 
   \[ \{Y_\sph, X_\sph, \Mm_\sph, \Mm^{\tor}_\sph,\quad Y_0(\p_1), X_0(\p_1), \Mm_0(\p_1), \Mm_0(\p_1)^{\tor}\}\]
   for the base-changes of these canonical models to $\Zp$. By construction $\Mm_\sph$ and $\Mm_0(\p_1)$ are moduli spaces for tuples $(A,\iota, \lambda, \alpha_{K^{(p)}})$, resp.~$(A,\iota, \lambda, \alpha_{K^{(p)}}, H)$, over $\Zp$-algebras.

   \begin{rmk}
    It should be possible to define integral models of the Shimura varieties $Y_0(\p_1^n)$, and their compactifications, for $n \ge 2$; but we are not aware of a reference where this is worked out in detail.
   \end{rmk}

  \subsubsection{Automorphic sheaves}

   In this section, let $\Mm$ denote either $\Mm_\sph$ or $\Mm_0(\p_1)$, and $\Mm^{\tor}$ its compactification. Let $\Aa \to \Mm^{\tor}$ be the semi-abelian scheme extending the universal abelian surface with real multiplication by $\Oo_F$. As in \cite[\S2.1]{myhida} this is built using the universal abelian schemes $\Aa^{\cc}\to \Mm^{\cc}$ for each $\cc$. Let $e: \Mm^{\tor}\to \Aa$ be the unit section and
   \[
    \myom_{\Aa}\coloneqq \mathop{e^{\star}} \Omega^1_{\Aa/\Mm^{\tor}} ;
   \]
   this is a $(\Oo_{\Mm^{\tor}}\otimes_{\Z}\Oo_F)$-module locally free of rank 1. We can write
   \[
    \myom_{\Aa} = \omega_{\Aa,1}\oplus \omega_{\Aa,2},
   \]
   where $\omega_{\Aa,i}$ is the direct summand on which the $\Oo_F$-action is given by the embedding $\sigma_i$. Let $\hh^1_{\Aa}$ be the canonical extension to $\Mm^{\tor}$ of the relative de Rham cohomology $\hh^1_{\dr}(\Aa/\Mm)$. It is a $(\Oo_{\Mm^{\tor}}\otimes_{\Z}\Oo_F)$-module locally free of rank 2 and we can decompose it with respect to the $\Oo_F$-action as above.

   For $(\kk, w) \in \Z^{2}\times \Z$ such that $k_{i}\equiv w \mod 2$, let
   \begin{equation}\label{eq:defiomegaGmodel}
    \myomk \coloneqq 
    \bigotimes_{1\le i\le 2}\left(\left(\wedge^{2} \hh_{\Aa,i}^{1}\right)^{\frac{w-k_{i}}{2}} \otimes {\omega}_{\Aa,i}^{k_{i}}\right)
   \end{equation}
   Using the descent datum of \cite[Definition 2.3.3]{myhida} (see also \cite[\S 5.5]{fakhruddinpilloni}) for the map $\Mm^{\tor}\to X$ (where $X$ is either $X_\sph$ or $X_0(\p_1)$), we can descend $\myomk$ to a sheaf over $X$. The cohomology of this sheaf has an action of the prime-to-$p$ Hecke algebra of level $K^{(p)}$, and the double coset of $\stbt{x}{0}{0}{x}$ for $x \in F^{\times +}$ acts as multiplication by $x \mapsto \nm_{F / \Q}(x)^w$.

   \begin{rmk}
    \label{rmk:changew}
    Note that the sheaf $\myom^{(0, 0, 2)} = \det \hh^1_{\Aa}$ has a canonical trivialization given by Poincar\'e duality. So if $w' = w \bmod 2$, the spaces $\RG(X, \myomk)$ and $\RG(X, \myom^{(\kk, w')})$ are canonically isomorphic as vector spaces; but this isomorphism twists the Hecke action by a power of the ad\`ele norm character.
   \end{rmk}

   If $E$ is any $F$-algebra, then we can similarly define $\myomk$ as a line bundle over $X_0(\p_1^n)_E$ for any $n$.

  \subsubsection{Geometry of the special fibre for $n=0$} 

   Recall that $X_{\sph}$ is a smooth $\Zp$-scheme. Its special fibre $X_{\sph,\Fp}$ can be decomposed as 
   \[
    X_{\sph,\Fp} = X_{\sph,\Fp}^{1-\ord} \cup D_{\p_1},
   \]
   where $D_{\p_1}$ is the vanishing locus of the partial Hasse invariant at $\p_1$ (descended from $\mathcal{M}^{\tor}_{\sph,\Fp}$ to $X_{\sph,\Fp}$ as in \cite[\S3.2]{myhida}) and is a reduced divisor on $X_{\sph,\Fp}$ not intersecting the toroidal boundary. Its preimage in $\Mm_{\sph, \Fp}$ corresponds to abelian surfaces which are supersingular at $\p_1$. We denote by $\X_\sph$ the formal completion of $X_\sph$ along $X_{\sph,\Fp}$ and $\X_\sph^{1-\ord}$ its $\p_1$-ordinary locus.

  \subsubsection{Geometry of the special fibre for $n=1$}

   Although $X_0(\p_1)_{\Fp}$ is not smooth, the $\p_1$-ordinary locus $X_0(\p_1)_{\Fp}^{1-\ord} \subset X_0(\p_1)_{\Fp}$ is smooth; it is the union of two components
   \[ X_0(\p_1)_{\Fp}^{1-\ord}= X_0(\p_1)_{\Fp}^{1-\m} \sqcup X_0(\p_1)_{\Fp}^{1-\et}, \]
   whose preimages in $\Mm_0(\p_1)_{\Fp}^{1-\ord}$ are the loci where the $\p_1$-level subgroup $H \subset A[\p_1]$ is multiplicative, resp.~\'etale. (See \cite[\S 4.1]{myhida}.)
   Denote by $\X_0(\p_1)$ the formal completion of $X_0(\p_1)$ along its special fibre; then we have a corresponding decomposition $\X_0(\p_1)^{1-\ord} = \X_0(\p_1)^{1-\m} \sqcup \X_0(\p_1)^{1-\et}$.

   Since a $\p_1$-ordinary abelian surface has a unique multiplicative $\p_1$-subgroup, restriction of the natural quotient map $X_0(\p_1) \to X_\sph$ is an isomorphism $\X_0(\p_1)^{1-\m} \to \X_{\sph}^{1-\ord}$. 

 \subsection{Classicality at spherical level}

  In this subsection we work on the surface at spherical level $X_\sph$: we will recall the definition of the normalised $T_{\p_1}$-operator of \cite[\S3.1]{myhida} and use the mod $p$-classicality of \emph{op.cit.} to deduce a classicality result in fixed weight at spherical level. %

  \subsubsection{The $T_{\p_1}$-operator} 

   We shall define a cohomological correspondence on $\RG(X_\sph, \myomk)$ using the diagram (cf.~\cite[\S 3.1]{myhida})
   \[
    \begin{tikzcd}[column sep=1em]
     & X_0(\p_1) \arrow{dr}[right]{p_2} \arrow{dl}[left]{p_1}\\
     X_\sph && X_\sph.
    \end{tikzcd}
   \]
   Here $p_1$ is the natural quotient map, whereas $p_2$ is induced by right-translation by $\stbt{1}{}{}{\varpi_{\p_1}}$ on $\GL_2(\AFf)$, for $\varpi_{\p_1}$ a uniformizer at $\p_1$. 

   \begin{rmk}
    The corresponding maps of moduli spaces $\Mm_0(\p_1) \to \Mm_{\sph}$ are given by
    \begin{align*}
     p_1: (A,\iota, \lambda, \alpha_{K^{(p)}}, H) &\mapsto (A, \iota, \lambda, \alpha_{K^{(p)}}),\\
     p_2 : (A,\iota, \lambda, \alpha_{K^{(p)}}, H) &\mapsto (A/H, \iota', \lambda', \alpha'_{K^{(p)}}),
    \end{align*} 
    where $\iota'$, $\lambda'$, $\alpha'_{K^{(p)}}$ are the induced endomorphism action, polarisation and level structure (see the discussion before \cite[(2.3)]{myhida}).
   \end{rmk}

   As in \emph{loc.cit.}, we can define a rational map $\pr_2^{\star}\myomk \dashrightarrow \pr_1^{\star}\myomk$ of sheaves on $X_0(\p_1)$ (using pullback along the isogeny $\pr_1^{\star}\Aa_\sph\to \pr_2^{\star}\Aa_\sph$ over the moduli space). Tensoring $\pr_2^{\star}{\omega}^{(\kk,w)} \dashrightarrow \pr_1^{\star}\myomk$ with the fundamental class $\Theta:  \pr_1^\star\Oo_{X_\sph}\to \pr_1^!\Oo_{X_\sph}$, we get the \emph{na\"ive} correspondence $T^{\naive}_{\p_1, (\kk,w)}: \pr_2^{\star}\myomk \dashrightarrow \pr_1^{!}\myomk$.
   We normalise it by setting
   \begin{displaymath}
    T_{\p_1, (\kk,w)} \coloneqq p^{-\inf\lbrace\tfrac{w-k_{1}}{2}+1,\tfrac{w+k_{1}}{2}\rbrace}T^{\naive}_{\p_1, (\kk,w)}=\begin{cases}
     p^{-\tfrac{w-k_{1}}{2}-1}T^{\naive}_{\p_1, (\kk,w)} &\text{if }k_{1}\geq 1 \\
     p^{-\tfrac{w+k_{1}}{2}}T^{\naive}_{\p_1, (\kk,w)} &\text{if }k_{1}< 1,
    \end{cases}
   \end{displaymath}
   To simplify the notation we will often denote simply by $T_{\p_1}$ the operator $T_{\p_1, (\kk,w)}$ for the automorphic sheaf $\myomk$.
   By \cite[Proposition 3.1.1]{myhida}, $T_{\p_1}$ extends to a map $\pr_2^{\star}\myomk \to \pr_1^{!}\myomk$ over the whole of $X_0(\p_1)$, and hence defines an endomorphism of $\RG\left(X_\sph, \myomk\right)$. Moreover this correspondence is optimally integral (does not vanish on the special fibre).

   \begin{rmk}
    \label{rmk:BGGdecomp}
    The integrality of $T_{\p_1}$ is an instance of a much more general result of Fakhruddin--Pilloni, see \cite[\S 5.6.4]{fakhruddinpilloni}. To motivate these choices of normalising factors, note that if $\Pi$ is a cuspidal automorphic representation of weight $(k_1, k_2)$, with $k_1, k_2 \ge 2$, and unramified at $p$, as in the introduction, then the representation $\Pi \otimes \|\cdot\|^{-w/2}$ appears in the cohomology of $X_{\overline{\Q}}$ with coefficients in each of the four sheaves $\myom^{(1 \pm (k_1 - 1), 1 \pm (k_2 - 1), w)}$ for small enough tame levels (this follows from the BGG decomposition for Hilbert modular varieties, see for example \cite[Example 2.6]{lansurvey}). For any of these choices of signs, the eigenvalues of $T^{\naive}_{\p_1}$ on this eigenspace are $p^{(1 + w)/2} \alpha_1^{\circ}$ and $p^{(1 + w)/2} \beta_1^{\circ}$, while the eigenvalues of $T_{\p_1}$ are $p^{(k_1 - 1)/2} \alpha_1^{\circ} = \alpha_1$ and $p^{(k_1 - 1)/2} \beta_1^{\circ} = \beta_1$. In particular, if we identify cohomology groups with the same $\kk$ but different $w$ via the isomorphism of \cref{rmk:changew}, the action of $T_{\p_1}$ is independent of $w$ (while the prime-to-$p$ Hecke action and $T^{\naive}_{\p_1}$ are not), since we are dividing by a power of $p$ which is removing the dependence on $w$ of the eigenvalues.
   \end{rmk}

  \subsubsection{Classicality for $T_{\p_1}$-ordinary cohomology}

   We denote by $e(T_{\p_1})$ the idempotent attached to the Hecke operator $T_{\p_1}$. 
   Applying \cite[Corollary 4.1.4]{myhida}, we obtain the following result:

   \begin{prop}
    If $k_1\geq 3$, then the natural restriction map induces a quasi-isomorphism:
    \[
     e(T_{\p_1})\RG\left(X_{\sph,\Fp},\myomk\right) \xrightarrow{\simeq} e(T_{\p_1})\RG\left(X_{\sph,\Fp}^{1-\ord},\myomk\right).
    \]
   \end{prop}

   We want to use the above result to deduce the characteristic zero analogue statement. %

   \begin{thm}\label{thm:sphericalclass}
    If $k_1\geq 3$, then the natural restriction map induces a quasi-isomorphism:
    \[
     e(T_{\p_1})\RG\left(X_\sph,\myomk\right) \xrightarrow{\simeq} e(T_{\p_1})\RG\left(\X_\sph^{1-\ord},\myomk\right).
    \]
   \end{thm}

   \begin{proof}
    The reduction mod $p$ of this map is a quasi-isomorphism by the previous proposition. Hence the result follows by Nakayama's lemma in the form of \cite[Proposition 2.2.2]{pilloni}.
   \end{proof}

  \subsubsection{Decomposition on the 1-ordinary locus} We now decompose the correspondence $T_{\p_1}$ over the $\p_1$-ordinary locus. Writing $\X_0(\p_1)^{1-\ord}$ as the union $\X_0(\p_1)^{1-\et} \sqcup \X_0(\p_1)^{1-\m}$ of its \'etale and multiplicative components, we can write
   \[ T^{\naive}_{\p_1} |_{\X_\sph^{1-\ord}} = T_{\p_1}^{\naive, \et} + T_{\p_1}^{\naive, \m}. \]

   \begin{prop}\label{def:Upoperator1}
    The correspondences $\mathfrak{U}_{\p_1} = p^{-(w-k_1)/2 - 1} T_{\p_1}^{\naive, \et}$ and $\mathfrak{V}_{\p_1} = p^{-(w + k_1)/2} T_{\p_1}^{\naive, \m}$ are defined integrally over $\X_\sph^{1-\ord}$ (for any $k_1$). Hence we have
    \[ T_{\p_1} |_{\X_\sph^{1-\ord}} =  
     \begin{cases}
      \mathfrak{U}_{\p_1} + p^{(k_1 - 1)} \mathfrak{V}_{\p_1} & \text{if $k_1 \ge 1$}, \\
      p^{(1-k_1)} \mathfrak{U}_{\p_1} + \mathfrak{V}_{\p_1} & \text{if $k_1 \le 1$}.
     \end{cases}
    \]
    In particular, over $\X_\sph^{1-\ord}$ we have $T_{\p_1} = \mathfrak{U}_{\p_1} \bmod p$ if $k_1 \ge 2$, and $T_{\p_1} = \mathfrak{V}_{\p_1} \bmod p$ if $k_1 \le 0$.
   \end{prop}

   \begin{proof} See \cite[Lemma 4.2.9]{myhida}. \end{proof}

   \begin{cor}
    For $k_1 \ge 2$, applying $e(\mathfrak{U}_{\p_1})$ gives a quasi-isomorphism
    \[ e(T_{\p_1})\RG\left(\X_\sph^{1-\ord},\myomk\right) \xrightarrow{\simeq}
     e(\mathfrak{U}_{\p_1})\RG\left(\X_\sph^{1-\ord},\myomk\right).\]
   \end{cor}

 \subsection{Family variation and Igusa towers}

  We now show that the automorphic line bundles can be interpolated into families, with $k_1$ varying, over the $\p_1$-ordinary locus. We suppose here that $R$ is the ring of integers of a finite extension of $\Qp$.

  \subsubsection{Families of sheaves}  

   \begin{thm}
    Let $(k_1, k_2, w) \in \Z^3$ be given with $k_1 = k_2 = w \bmod 2$; and let $\Lambda = R[[\Zp^\times]]$ with its canonical character $\mathbf{\kappa}: \Zp^\times \to \Lambda^\times$. Then there exists an invertible sheaf of $\Lambda \otimes \Oo_{\X^{1-\ord}_\sph}$-modules $\Omega^{(k_1 + 2\kappa, k_2, w)}$, whose specialisation at the character $x \mapsto x^n$ of $\Lambda$, for any $n \in \Z$, is $\myom^{(k_1 + 2n, k_2, w)}|_{\X^{1-\ord}_\sph}$.
   \end{thm}

   We first work over the moduli space $\M^{\tor, 1-\ord}_{\sph}$, the formal completion of the $\p_1$-ordinary locus of $\Mm_\sph^{\tor}$. Similarly as in \cite{myhida}, we can consider the $\Zp^{\times}$-torsor
   \[
    \Ig(\p_1^{\infty}) \coloneqq \operatorname{Isom}(\mu_{p^\infty}, \Aa[\p_1^\infty]^\circ)\xrightarrow{\pi} \M^{\tor, 1-\ord}_\sph,
   \]
   where $\Aa[\p_1^\infty]^\circ$ denotes the connected part of the $\p_1$-torsion of the semiabelian variety $\Aa$ over $\M^{\tor, 1-\ord}_{\sph}$.
   We then define the sheaf
   \[
    \Omega_{1}^{\kappa}\coloneqq(\pi_{\star}\Oo_{\Ig(\p_1^{\infty})} \mathop{\hat\otimes} \Lambda)^{\Zp^{\times}},
   \]
   where $\Zp^{\times}$ acts on $\Lambda$ via the universal character $\kappa: \Zp^{\times}\to \Lambda$, and on $\pi_{\star}\Oo_{\Ig(\p_1^{\infty})}$ via the action on $\Ig(\p_1^{\infty})$. As in \cite[Lemma 4.2.1]{myhida}, for any $k \in \Z$ the specialisation of this sheaf at the character $x \mapsto x^k$ of $\Lambda$ is given by
   \[
    \Omega_{1}^{\kappa}\otimes_{k}R \simeq \omega_{\Aa,1}^{k}.
   \]

   We also need to interpolate the powers of $\wedge^2 \hh^1_{\Aa,1}$. Similarly as in \cite[Definition 4.2.2]{myhida} (where instead the Igusa varieties at all primes above $p$ are considered), we can also introduce the $\Zp^\times$-torsor $\Ig^{\vee}(\p_1^\infty)\coloneqq\text{Isom}(\mu_{p^\infty}, \Aa^\vee[\p_1^\infty]^{\circ})\xrightarrow{\pi^\vee} \M^{\tor, 1-\ord}_\sph$, where $\Aa^\vee$ denotes the dual of the universal abelian variety $\Aa$ (which again extends to a semiabelian scheme over toroidal compactification).

   Consider
   \[
    \left( \wedge^2 \hh^1_{\Aa, 1}\right)^{\kappa} \coloneqq 
    \left((\pi\times \pi^\vee)_{\star}\Oo_{\Ig(\p_1^\infty) \times \Ig^{\vee}(\p_1^\infty)} \mathop{\hat\otimes} \Lambda\right)^{\T},
   \]
   where $\T$ denotes the maximal torus of $\GL_2(\Zp)$, acting via its natural action on each component of $(\pi\times \pi^\vee)_{\star}\Oo_{\Ig(\p_1^\infty)\times \Ig^{\vee}(\p_1^\infty)}$, and on $\Lambda$ via the composite of the determinant $\T \to \Zp^\times$ and the universal character. One can show (interpreting $\wedge^2 \hh^1_{\Aa,1}$ in terms of a torsor as in \cite[\S2.3.1]{myhida}) that the specialisation of the above sheaf at the character $x \mapsto x^k$ of $\Lambda$ is given by
   \[ 
    \left( \wedge^2 \hh^1_{\Aa, 1}\right)^{\kappa}\otimes_\Lambda R \simeq (\wedge^2 \hh^1_{\Aa,1})^k,   
   \]
   justifying the notation. We now define
   \[
    \Omega^{(2\kappa, 0, 0)} \coloneqq
    \Omega_{1}^{2\kappa} \otimes \left( \wedge^2 \hh^1_{\Aa, 1}\right)^{-\kappa},
   \]
   which is a sheaf of $\Lambda \otimes \Oo_{\M^{\tor, 1-\ord}_\sph}$-modules; by construction we have
   \begin{equation}
    \label{eq:specialiso}
    \Omega^{(2\kappa, 0, 0)} \otimes_{\Lambda, n} R = \myom^{(2n, 0, 0)}
   \end{equation}
   as sheaves on $\M^{\tor, 1-\ord}_\sph$. Exactly as in \cite[\S 4.2.2]{myhida}, the unit group $\Delta_{K^{(p)}}$ acts trivially on $\Omega^{(2\kappa, 0, 0)}$ so it descends to a sheaf on $\X^{1-\ord}_{\sph}$.

   Finally we consider $k_1 \equiv k_2 \equiv w \mod 2$ and consider the sheaf 
   \(
   \Omega^{(k_1 + 2\kappa, k_2, w)} \coloneqq
   \Omega^{(2\kappa, 0, 0)} \otimes_{\Oo_{\X_\sph^{1-\ord}}} \myomk.
   \)
   This way the isomorphism \eqref{eq:specialiso} induces an isomorphism
   \begin{equation} 
    \label{eq:specialiso2}
    \Omega^{(k_1 + 2\kappa, k_2, w)}\otimes_{\Lambda, a} R \simeq (\myom^{(k_1+2a,k_2,w)})|_{\X_\sph^{1-\ord}}
   \end{equation}
   for all $a \in \Z$.

   \begin{rmk}
    Note that if $k_1' = k_1 + 2t$ for some integer $t$, then $\Omega^{(k_1' + 2\kappa, k_2, w)}$ is isomorphic to the base-change of $\Omega^{(k_1 + 2\kappa, k_2, w)}$ via the map $[x] \mapsto x^t [x]$ on $\Lambda$.
   \end{rmk}

  \subsubsection{Hecke action}

   Let us fix as above $k_1,k_2,w$ such that $k_1\equiv k_2\equiv w \mod 2$ and consider the sheaf of $\Lambda\otimes\Oo_{\X_\sph^{1-\ord}}$-modules $\Omega^{(k_1 + 2\kappa, k_2, w)}$.

   \begin{defi}
    We denote by $\mathfrak{U}_{\p_1}$ the Hecke operator acting on $\RG\left(\X^{1-\ord}_\sph, \Omega^{(k_1 + 2\kappa, k_2, w)}\right)$ defined as in \cite[Lemma 4.2.7]{myhida} (the operator denoted $U_{\p_1}$ in \emph{op.cit.}).
   \end{defi}

   As shown in Lemma 4.2.8 of \emph{op.cit.}, this is compatible under specialization at $\kappa = a$, for any $a \in \Z$, with the operator $\mathfrak{U}_{\p_1}$ defined above on $\myom^{(k_1 + 2a, k_2, w)}$. 

   \begin{rmk} 
    In \emph{op.cit.} this operator is defined, and its specialization property shown, for the full Igusa tower (trivializing both $\p_1$ and $\p_2$ torsion, and varying both components of the weight) over the fully ordinary locus $\X_\sph^{\ord}$; but the same argument works over $\X_\sph^{1-\ord}$ for the partial Igusa tower at $\p_1$.
   \end{rmk}

   We can therefore consider the complex of $\Lambda$-modules
   \[
    M^\bullet(k_1 + 2\kappa, k_2, w) \coloneqq
    e(\mathfrak{U}_{\p_1})\RG\left(\X_\sph^{1-\ord},\Omega^{(k_1 + 2\kappa, k_2, w)}\right).
   \]

   \begin{thm}\label{thm:classsphcomplex}
    The complex $M^\bullet(k_1 + 2\kappa, k_2, w)$ %
    is a perfect complex of $\Lambda$-modules concentrated in degrees $[0, 2]$, and for any $a\in\Z$, the isomorphism of sheaves \eqref{eq:specialiso2} yields a quasi-isomorphism
    \begin{align*}
     M^\bullet(k_1 + 2\kappa, k_2, w) \otimes^{\mathbb{L}}_{\Lambda,a} R 
     &\cong
     e(\mathfrak{U}_{\p_1}) \RG\left(\X_\sph^{1-\ord},\myom^{(k_1+2a,k_2,w)}\right)%
    \end{align*}
    compatible with the action of $\mathfrak{U}_{\p_1}$ and the Hecke operators away from $p$.
   \end{thm}

   \begin{proof}
    This follows exactly as in Theorem 4.2.13 of \cite{myhida}.
   \end{proof}

   \begin{rmk}\label{rmk:onedegree}
    Writing $D$ for the boundary divisor, one can show similarly that the complex
    \[ M^\bullet(k_1 + 2\kappa, k_2, w; -D)\coloneqq e(\mathfrak{U}_{\p_1})\RG\left(\X_\sph^{1-\ord},\Omega^{(k_1 + 2\kappa, k_2, w)}(-D)\right) \]
    is perfect, and interpolates the spaces $e(\mathfrak{U}_{\p_1}) \RG\left(\X_\sph^{1-\ord},\myom^{(k_1+2a,k_2,w)}(-D)\right)$. Moreover, as in Proposition 4.2.15 of \cite{myhida}, one can show that
    \begin{itemize}
     \item the natural map $M^\bullet(k_1 + 2\kappa, k_2, w; -D) \to M^\bullet(k_1 + 2\kappa, k_2, w)$ is a quasi-isomorphism after localisation at a non-Eisenstein maximal ideal of the Hecke algebra;
     \item the complex $M^\bullet(k_1 + 2\kappa, k_2, w, -D)$ is concentrated in degrees $[0, 1]$.
    \end{itemize}
    Moreover, if $k_2 \le -1$, then the vanishing results of \cite{diamondkassaei17} show that we have
    \[ H^0(X_{\sph, \Fp}, \myomk) = H^0(X_{\sph, \Fp}, \myomk(-D)) = 0 \]
    for any $k_1$. Applying this with $k_1$ replaced by $k_1 + 2a$, for $a \gg 0$, and using the classicity result of \cref{thm:sphericalclass}, it follows that $M^\bullet(k_1 + 2\kappa, k_2, w)$ is represented by a perfect complex concentrated in degrees $[1, 2]$, and $M^\bullet(k_1 + 2\kappa, k_2, w; -D)$ by a single projective $\Lambda$-module in degree 1 alone. Similarly, one can show that if $k_2 \ge 3$ then $M^\bullet(k_1 + 2\kappa, k_2, w; -D)$ is represented by a single projective $\Lambda$-module in degree 0.
   \end{rmk}

 \subsection{Interpolation at Iwahori level} 

  Note that our classicality theorem \cref{thm:sphericalclass} uses the $T_{\p_1}$-ordinary cohomology, but it is the $\mathfrak{U}_{\p_1}$-ordinary part which interpolates in families; the operator $T_{\p_1}$ does not make sense on the Igusa tower. In this section we shall give an alternative formulation of the classicality theorem at level $X_0(\p_1)$, which is more convenient for applications.

  \subsubsection{The $U_{\p_1}$-operator on the generic-fibre}
   For $n \ge 1$, let $K_{\p_1, n}'$ denote the subgroup $\{ \stbt a b c d : p \mid b, p^n \mid c\}$ of $\GL_2(\Oo_{F, \p_1})$. Let $C(\p_1^n)_\Q$ denote the compactified Shimura variety for $G$ of level $K_{\p_1, n}' K^{(\p_1)}$, so we have a diagram of degeneracy maps
   \[
    \begin{tikzcd}[column sep=1em]
     & C(\p_1^n)_\Q \arrow{dr}[above]{p'_2} \arrow{dl}[above]{p'_1}\\
     X_0(\p_1^n)_\Q && X_0(\p_1^n)_\Q.
    \end{tikzcd}
   \]
   where $p'_1$ is the natural projection, and $p'_2$ is given by the right action of $\stbt{\varpi_{\p_1}}{}{}{1}$ for $\varpi_{\p_1}$ a uniformizer at $\p_1$. We can and do choose our toroidal boundary data so that $C(\p_1^n)_{\Q}$ is smooth and projective, and the maps $p'_i$ are well-defined and finite. (The moduli space associated to $C(\p_1^n)$ parametrises choices of a second degree $p$ subgroup $H' \subset A[\p_1]$, disjoint from the $p^n$-subgroup $H$ parametrized by $\mathcal{M}_0(\p_1^n)$; and the maps $p'_1, p'_2$ correspond to forgetting, resp.~quotienting out by, $H'$.)

   Let $E$ be an arbitrary $F$-algebra (so that our automorphic sheaves are defined). Then, by general results of Harris (see \cite[\S 4.3.3]{harrisdelta}), pullback along the universal isogeny $p_1^{\prime*}(\Aa) \to p_2^{\prime*}(\Aa) = \Aa / H'$ descends to a map (indeed an isomorphism) of vector bundles on $C(\p_1^n)_E$,
   \[ g^* : p_2^{\prime*} \left(\myomk\right) \to p_1^{\prime*} \left(\myomk\right), \]
   The composite $(p'_1)_* \circ g^* \circ (p'_2)^*$ defines an endomorphism of $\RG\left(X_0(\p_1^n)_E, \myomk\right)$ which we denote by $U_{\p_1, (\kk, w)}^{\naive}$. We finally normalise it by setting
   \[
    U_{\p_1, (\kk,w)}\coloneqq p^{(k_1 - w - 2)/2}U^{\naive}_{\p_1, (\kk,w)}.
   \]
   We drop the subscript $(\kk, w)$ if it is clear from context.

   \begin{rmk}
    Note we don't claim that $U_{\p_1}$ is defined integrally; this would not be a well-defined statement since we have not fixed an integral model of $C(\p_1^n)$.
   \end{rmk}

  \subsubsection{Comparison of $U_{\p_1}$ and $\mathfrak{U}_{\p_1}$}

   We now take $n = 1$, and consider the rigid-analytic generic fibre $\X_0(\p_1)_{\Qp}^{1-\m}$ of the formal scheme $\X_0(\p_1)^{1-\m}$, which is an open rigid-analytic subvariety of the analytification $\X_0(\p_1)_{\Qp} = \left(X_0(\p_1)_{\Qp}\right)^{\mathrm{an}}$. Since the formal scheme $\X_0(\p_1)^{1-\m}$ is of finite type over $\operatorname{Spf} \Zp$ we have 
   \[ 
    \RG\left(\X_0(\p_1)_{\Qp}^{1-\m}, \myomk\right) = 
    \RG\left(\X_0(\p_1)^{1-\m}, \myomk\right) \otimes_{\Zp} \Qp.
   \]
   We then have two operators on this space:
   \begin{itemize}
    \item The Hecke operator $\mathfrak{U}_{\p_1}$ of \cref{def:Upoperator1}, transported from $\X_\sph^{1-\ord}$ via the isomorphism $p_1 : \X_0(\p_1)^{1-\m} \to \X_\sph^{1-\ord}$.
    \item The Hecke operator $U_{\p_1}$ defined above, which restricts to a correspondence $\X_0(\p_1)_{\Qp}^{1-\m} \rightrightarrows \X_0(\p_1)_{\Qp}^{1-\m}$ (since the subgroup $H'$ parametrised by $C(\p_1)$ is disjoint from $H$, so the image of $H$ in $A / H'$ is multiplicative iff $H$ itself is). 
   \end{itemize}

   \begin{prop}
    These two Hecke operators coincide.
   \end{prop}

   \begin{proof}
    It suffices to prove the corresponding equality over the open moduli space. We constructed $U_{\p_1}$ using the moduli of $\p_1$-subgroups $H'$ complementary to the level subgroup $H$. Over $\X_0(\p_1)^{1-\m}$, the level subgroup $H$ is the unique multiplicative subgroup of $A[\p_1]$; so the subgroups $H'$ complementary to $H$ are precisely the \'etale subgroups of $A[\p_1]$, which are classified by the correspondence $\mathfrak{U}_{\p_1}$ above. 
   \end{proof}

  \subsubsection{Comparison of classical cohomology at spherical and Iwahori level}

   We now show that the right hand side of the isomorphism of Theorem \ref{thm:classsphcomplex} is isomorphic to the $U_{\p_1}$-ordinary part of the cohomology of the surface at Iwahori level $X_0(\p_1)$, at least over $\Qp$. In particular this will allow us to show that the rational complex defined at spherical level interpolates cohomology at Iwahori level.

   \begin{thm}
    Suppose $k_1 \ge 2$, and let $\Pi_{\p_1}$ be an irreducible subquotient of the $\GL_2(\Qp)$-representation 
    \[ \colim_{\substack{U \subset \GL_2(K_{\p_1}) \\ \text{open compact} }} H^i\left(X(K^{(\p_1)} U)_{\overline{\Q}_p}, \myomk\right) \]
    such that $e_{U_{\p_1}} \cdot (\Pi_{\p_1})^{K_{\p_1, 1}} \ne 0$. (Here $X(K)$ denotes a toroidal compactification of the Shimura variety for $G$ of level $K \subset \GL_2(\AFf)$.) Then $\Pi_{\p_1}$ is infinite-dimensional (equivalently, generic), and $e_{U_{\p_1}} \cdot (\Pi_{\p_1})^{K_{\p_1, 1}}$ has dimension 1. Moreover, exactly one of the following occurs:
    \begin{itemize}
     \item $\Pi_{\p_1}$ is an irreducible unramified principal series, and the integrally-normalized Hecke parameters satisfy $v_p(\alpha_1) = 0$, $v_p(\beta_1) = k_1 - 1$, with $\alpha_1$ being the eigenvalue of $U_{\p_1}$ on $e_{U_{\p_1}} \cdot (\Pi_{\p_1})^{K_{\p_1, 1}}$.
     \item $k_1 = 2$, and $\Pi_{\p_1}$ is an unramified twist of the Steinberg representation, the unique infinite-dimensional constituent of the reducible principal series with Hecke parameters $\alpha_1, \beta_1$ satisfying $v_p(\alpha_1) = 0$, $v_p(\beta_1) = 1$, and $\beta_1 = p \alpha_1$. Moreover, $\alpha_1$ is the eigenvalue of $U_{\p_1}$ on $e_{U_{\p_1}} \cdot (\Pi_{\p_1})^{K_{\p_1, 1}}$.
    \end{itemize}
   \end{thm}

   \begin{proof} 
    Since $\Pi_{\p_1}$ has Iwahori-invariants, it must be either an irreducible unramified principal series, a 1-dimensional unramified representation, or an unramified-character twist of the Steinberg. If it is an unramified principal series, then the normalized Hecke parameters satisfy $v_p(\alpha), v_p(\beta) \ge 0$ and $v_p(\alpha) + v_p(\beta) = k_1 - 1$, and the $U_{\p_1}$ eigenvalues are $\alpha$ and $\beta$, so the only possibility is as described.

    In the one-dimensional case, $\Pi_{\p_1} = \eta \circ \det$ for an unramified character $\eta$, and $\eta(p)^2 = p^{(k_1-2)}$ times a root of unity; but the $U_{\p_1}$-eigenvalue on the Iwahori invariants is then $p \eta(p)$, which has valuation $k_1 / 2 \ge 1$, so it cannot be a $p$-adic unit. In the Steinberg case, a similar computation gives that the $U_{\p_1}$ eigenvalue has valuation $k_1/2 - 1$, so this case can only occur when $k_1 = 2$ (this is just the well-known fact that $p$-new modular forms are ordinary in weight 2 and are non-ordinary in any higher weight).
   \end{proof}

   In particular, if $k_1 \ge 3$, then any $\Pi_{\p_1}$ as above must be an unramified irreducible principal series, and the natural map
   \[ 
    \Pi_{\p_1}^{\GL_2(\Zp)} \hookrightarrow \Pi_{\p_1}^{K_{\p_1, 1}} \stackrel{e(U_{\p_1})}{\relbar\joinrel\relbar\joinrel\twoheadrightarrow} e(U_{\p_1}) \Pi_{\p_1}^{K_{\p_1, 1}}
   \]
   is an isomorphism of one-dimensional $\overline{\Q}_p$-vector spaces. Hence we have the following corollary:

   \begin{cor}\label{cor:sphiwahoricohomology} 
    If $k_1 \ge 3$, then the composite of pullback to level $K_{1}$ and the operator $e(U_{\p_1})$ is an isomorphism
    \[ e(T_{\p_1}) H^i(X_{\sph,\Qp}, \myomk) \xrightarrow{\ \cong\ } e(U_{\p_1}) H^i(X_0(\p_1)_{\Qp}, \myomk). \]
   \end{cor}

  \subsubsection{Final result} 

   We now consider the following diagram:
   \[
    \begin{tikzcd}[column sep=1em]
     e(T_{\p_1}) \RG(X_{\sph, E},\myomk) 
     \arrow{r}{} \arrow{d} 
     & e(T_{\p_1}) \RG(\X_{\sph, E}^{1-\ord}, \myomk) \dar\\
     e(U_{\p_1}) \RG(X_0(\p_1)_{E}, \myomk) 
     \arrow{r}{} 
     & e(U_{\p_1})\RG(\X_0(\p_1)_E^{1-\m},\myomk).
    \end{tikzcd}
   \]
   Here the maps are as follows:
   \begin{itemize}
    \item The horizontal arrows are restriction along the open embeddings $\X_{\sph,E}^{1-\ord} \subset \X_{\sph,E}$ and $\X_0(\p_1)_{E}^{1-\m} \subset \X_0(\p_1)_E$.
    \item The vertical arrows are the composite of the natural pullback $p_1^\star$ and the idempotent $e(U_{\p_1})$.
   \end{itemize}

   If $k_1 \ge 3$, then the top horizontal arrow is a quasi-isomorphism by \cref{thm:classsphcomplex}; the left vertical arrow is a quasi-isomorphism by \cref{cor:sphiwahoricohomology}; and the right vertical arrow is also a quasi-isomorphism, since $e(T_{\p_1})$ and $e(\mathfrak{U}_{\p_1})$ agree modulo $p$, so the map is the identity modulo $p$ by \cref{def:Upoperator1}, and hence a quasi-isomorphism by Nakayama's lemma. Thus the lower horizontal arrow is a quasi-isomorphism as well, and we deduce the following:

   \begin{cor}\label{thm:classicalityigusa}
    For any $a\in\Z$ such that $k_1+2a\geq 3$, the composite of restriction to $\X_0(\p_1)^{1-\m}_E$ and \cref{eq:specialiso2} defines a quasi-isomorphism
    \begin{align*}
     e(U_{\p_1}) \RG\left(X_0(\p_1)_E,\myom^{(k_1+2a,k_2,w)}\right) &\xrightarrow{\simeq} M^\bullet(\kk, w) \otimes^{\mathbb{L}}_{\Lambda,a} E.%
     \tag*{\qedsymbol}
    \end{align*} 
   \end{cor}

 \subsection{Hecke operators at deeper Iwahori level} 

  Note that in the previous sections we worked at Iwahori level. As the $\GL_2$ $p$-adic Eisenstein series we will work with will have deeper level at $p$, we need to understand the action of the $U_{\p_1}$-operator on the cohomology of the Shimura variety $X_0(\p_1^n)$. Since the geometry of integral models of $X_0(\p_1^n)$ for $n > 1$ is complicated, we shall only consider this variety after inverting $p$. 

  \subsubsection{Pullback}

   We denote by $\pi_n:X_0(\p_1^n) \to X_0(\p_1)$ the natural projection map.
   \begin{prop}
    For general $n \ge 2$, the action of $(U_{\p_1})^{n-1}$ on the cohomology of $X_0(\p_1^n)_{\Qp}$ factors through the image via $(\pi_n)^\star$ of the cohomology of $X_0(\p_1)_{\Qp}$.
   \end{prop}

   \begin{proof}
    This is well-known. It suffices to check that $U_{\p_1}$ factors through $X_0(\p_1^{(n-1)})_{\Q}$, and this follows from the fact that the map $p_2' : C(\p_1^n)_{\Q} \to X_0(\p_1^n)_{\Q}$ actually factors through $C(\p_1^{(n-1)})_{\Q}$.
   \end{proof}

   It follows that the pullback map $\pi_n^\star$ induces a quasi-isomorphism
   \[ e(U_{\p_1}) R\Gamma(X_0(\p_1)_{\Qp}, \myomk) \to  e(U_{\p_1}) R\Gamma(X_0(\p_1^n)_{\Qp}, \myomk), \]
   and likewise for any finite extension $E / \Qp$.

  \subsubsection{Duality}
   Since $\omega^{(2, 2, 0)}(-D)$ is the dualising sheaf, for any $(k_1, k_2, w)$, any $n \ge 1$ and any field extension $E/\Q$, we have a perfect Serre duality pairing
   \[ H^1\left(X_0(\p_1^n)_E, \myom^{(k_1, 2-k_2, w)}\right) \times 
    H^1\left(X_0(\p_1^n)_E, \myom^{(2-k_1, k_2, -w)}(-D)\right) \to E. \]
   (We write the weights in this form so that if $k_1, k_2 \ge 1$, then the Hecke eigenvalue systems associated to holomorphic modular forms of weight $(k_1, k_2)$ contribute to both sides, cf.~\cref{rmk:BGGdecomp} above.)

   \begin{defi}
    We define the Hecke operator $U_{\p_1}'$ on $R\Gamma\left(X_0(\p_1^n)_E, \myom^{(2-k_1, k_2, -w)}(-D)\right)$ by
    \[ U_{\p_1}' = p^{(k_1 + w - 2)/2} \cdot [\stbt{1}{0}{0}{\varpi_{\p_1}}] \]
    where $[g]$ denotes the double coset of $g$.
   \end{defi} 

   Note that the transpose under the Serre duality pairing of the operator $U_{\p_1}$ on $\RG\left(X_0(\p_1^n)_E, \myom^{(k_1, 2-k_2, w)}\right)$ (the operator denoted by $U_{\p_1}^t$ in \cite{myhida}) is $\langle \p_1 \rangle^{-1} U_{\p_1}'$, where $\langle \p_1 \rangle = p^{w} \cdot [ \stbt{\varpi_{\p_1}}{}{}{\varpi_{\p_1}}]$ is a finite-order endomorphism in the centre of the Hecke algebra. In particular, assuming (for the rest of this section) that $E$ is a finite extension of $\Qp$, Serre duality restricts to a perfect pairing
   \[ 
    e(U_{\p_1}) H^1\left(X_0(\p_1^n)_E, \myom^{(k_1, 2-k_2, w)}\right) \times e(U_{\p_1}') H^1\left(X_0(\p_1^n)_E, \myom^{(2-k_1, k_2, -w)}(-D)\right) \to E.
   \]

  \subsubsection{Pairing with eigenclasses}

   We now consider a cuspidal automorphic representation $\Pi$ of $\GL_2(\A_F)$ of tame level $K^{(\p_1)}$ and unramified and ordinary at $\p_1$, generated by a holomorphic cusp form of weight $(k_1, k_2)$, so that (for large enough $E$) $\Pif \otimes |\cdot|^{w/2}$ contributes to the degree one cohomology of the sheaf $\myom^{(2-k_1,k_2,-w)}(-D)$, i.e.
   \[
    H^1(X_0(\p_1)_E,\myom^{(2-k_1,k_2,-w)}(-D))[\Pif]\neq 0.
   \]

   Similarly as in \cite[Definition 5.8]{padicLfnct}, we have the following.

   \begin{defi}\label{defi:etan}
    Let $\eta=\eta_1\in H^1(X_0(\p_1)_E,\myom^{(2-k_1,k_2,-w)}(-D))[\Pif]$ and let $n\ge 1$. We let
    \[ 
     \eta_n\in H^1(X_0(\p_1^n)_E,\myom^{(2-k_1,k_2,-w)}(-D))[\Pif]
    \]
    be the unique element such that the pairing
    \[
     \langle-,\eta_n\rangle : H^1(X_0(\p_1^n)_E,\myom^{(k_1,2-k_2,w)})\to E
    \]
    vanishes on the $U_{\p_1}=0$ eigenspace and it agrees with $p^{n-1}\langle-,\eta_1\rangle$ on the image of $H^1(X_0(\p_1)_{E},\myom^{(k_1,2-k_2,w)})$, i.e. the following diagram commutes:
    \begin{equation}
     \label{eq:diagrameta1etan}
     \begin{tikzcd}[column sep =27mm]
      H^1(X_0(\p_1^n)_E,\myom^{(k_1,2-k_2,w)})
      \arrow[rd, "{p^{1-n} \langle -, \eta_n\rangle}", start anchor = east, bend left=15]\\
      H^1(X_0(\p_1)_E,\myom^{(k_1,2-k_2,w)})
      \arrow{u}{\pi_n^\star}
      \arrow[r, "{\langle-,\eta_1\rangle }" below] & E.
     \end{tikzcd}
    \end{equation}
   \end{defi}

   The factor $p^{n-1}$ is the degree of the map $\pi_n$. It follows from the definition that the $\eta_n$'s are compatible under the trace map as follows:
   \[
    \eta_n=p^{-1}\cdot \sum_{g\in K_{\p_1, n} / K_{\p_1, n+1}}g\cdot \eta_{n+1}
   \]

   \begin{prop}
    If $\eta_1$ lies in the $U_{\p_1}'=\alpha$ eigenspace in the cohomology of $X_0(\p_1)_E$, then for any $n \ge 1$, $\eta_n$ is an eigenvector for $U_{\p_1}'$ in the cohomology of $X_0(\p_1^n)_E$, with the same eigenvalue.
   \end{prop}

   \begin{proof}
    This follows as in \cite[Proposition 5.9]{padicLfnct}.
   \end{proof}

   Via the ``canonical subgroup'' map, we can identify the rigid-analytic generic fibre $\X_0(\p_1)^{1-\m}_E$ with an open subspace of the analytification of $X_0(\p_1^n)_{E}$, for any $n \ge 1$. This gives a map
   \[
    H^1(X_0(\p_1^n)_E,\myom^{(k_1,2-k_2,w)}) \to H^1(\X_0(\p_1)^{1-\m}_E,\myom^{(k_1,2-k_2,w)}).
   \]

   \begin{prop}
    If $\eta_1$ lies in the $U_{\p_1}'=\alpha$ eigenspace, and $\alpha$ is a $p$-adic unit, then the map
    \begin{align*}
     H^1(\X_0(\p_1)^{1-\m}_E, \myom^{(k_1,2-k_2,w)})
     \xrightarrow{e(U_{\p_1})}&\, e(U_{\p_1})
     H^1(\X_0(\p_1)^{1-\m}_E, \myom^{(k_1,2-k_2,w)}) \\
     \cong\ & e(U_{\p_1}) H^1(X_0(\p_1)_E,\myom^{(k_1,2-k_2,w)})\\
     \xrightarrow{\langle -, \eta_1\rangle}&\, E
    \end{align*}
    coincides with $p^{1-n}\langle -, \eta_n\rangle$ on the image of $H^1(X_0(\p_1^n)_E,\myom^{(k_1,2-k_2,w)})$, for any $n \ge 1$.
   \end{prop}

   \begin{proof}
    This follows from the fact that $U_{\p_1}^{n-1}$ factors through $\pi_n^\star$.
   \end{proof}

   %%%%%%%%%%%%%%%%%%%%%%%%%%%%%%%%%%%%%%%%%%%%

 \subsection{Families of eigenclasses}
  \label{sect:eigenfamilies}
  As above we let $E$ be a finite extension of $\Qp$; and we now suppose that the tame level has the form $K_1(\mathfrak{N})^{(p)} = \{g \in \GL_2(\widehat{\Oo}_F^{(p)}) :  g = \stbt{*}{*}{0}{1} \bmod \mathfrak{N}\}$ for some ideal $\mathfrak{N}$ of $\Oo_F$.

  \begin{notation}
   Let $\T$ denote the product of the prime-to-$\mathfrak{N}\p_1$ Hecke algebra $\T^{(\mathfrak{N}\p_1)}$ (with $E$-coefficients) and the polynomial ring in a formal variable $U_{\p_1}$ and its inverse.
  \end{notation}

  \begin{defi}
   Let $\Uu \subset \Ww =(\operatorname{Spf} \Lambda)^{\mathrm{rig}}_E$ be an open affinoid containing 0; and let $\kappa_{\Uu} : \Zp^\times \to \Oo(\Uu)$ be the universal weight. A \emph{$\p_1$-adic family of eigensystems} over $\Uu$ of weight $(k_1 + 2\kappa_\Uu, k_2, w)$ and level $\mathfrak{N}$ is a homomorphism $\T \to \Oo(\Uu)$ with the following property:
   \begin{itemize}
    \item For every $a \in \Uu \cap \Z$ with $k_1 + 2a \ge 2$, the composite homomorphism $\T \to \Oo(\Uu) \to E$ given by evaluation at $a$ is the Hecke eigenvalue system associated to a $\p_1$-stabilisation of $\Pi[a]$ (or more precisely $\Pi[a] \otimes \|\cdot\|^{-w/2}$), for some cuspidal automorphic representation $\Pi[a]$ of $\GL_2 / F$ of level $\mathfrak{N}$ and weight $(k_1 + 2a, k_2)$ and some embedding of its coefficient field into $E$.
   \end{itemize}
   We denote such a family by $\underline{\Pi}$. We say that $\underline{\Pi}$ is \emph{$\p_1$-ordinary} if the Hecke eigenvalues of $U_{\p_1}$ and its inverse acting on $\underline{\Pi}$ are power-bounded.
  \end{defi}

  \begin{thm}   
   \label{thm:families}
   Let $\Pi$ be a cuspidal automorphic representation of conductor $\mathfrak{N}$ and weight $(k_1, k_2)$, with $k_1 = k_2 = w \bmod 2$, ordinary at $\p_1$, whose coefficient field is contained in $E$. Assume $k_1 \ge 3$ and $k_2 \ge 2$. Then:
   \begin{enumerate}[(i)]
    \item There exists an affinoid disc $\Uu \ni 0$ in $\Ww$, and a family $\underline{\Pi}$ over $\Uu$ passing through $\Pi$ (i.e.~with $\Pi[0] = \Pi$). Moreover, this family is uniquely determined.
    \item If we write
    \[ M^\bullet(k_1 + 2\kappa_{\Uu}, 2-k_2, w) \coloneqq M^\bullet(k_1 + 2\kappa, 2-k_2, w)  \otimes_{\Lambda} \Oo(\Uu), \]
    then, after possibly shrinking $\Uu$, the $H^1$ of this complex contains a free rank 1 $\Oo(\Uu)$-direct summand
    \[ H^1(k_1 + 2\kappa_{\Uu}, 2-k_2, w)[\underline{\Pi}] \ \subseteq\  H^1\left(M^\bullet(k_1 + 2\kappa_{\Uu}, 2-k_2, w)\right)\]
    whose specialisation at every $a \in \Uu \cap \Z$ with $k_1 + 2a \ge 3$ is identified via the isomorphism of \cref{thm:classicalityigusa} with the eigenspace in $e(U_{\p_1}) H^1(X_0(\p_1), \myom^{(k_1 + 2a, 2-k_2, w)})$ on which $\T$ acts via the ordinary $\p_1$-stabilisation of $\Pi[a]$.
   \end{enumerate}
  \end{thm}

  \begin{proof}
   Let $\Uu$ be an arbitrary (for now) affinoid disc around 0. We write $H^i_{\Uu}$ for the cohomology of $M^\bullet(k_1 + 2\kappa_{\Uu}, 2-k_2, w)$, and $\T_{\Uu}$ for the image of $\Oo(\Uu) \otimes \T$ in the $\Oo(\Uu)$-endomorphism algebra of the graded module $\bigoplus_i H^i_{\Uu}$; this is a finite $\Oo(\Uu)$-algebra, hence affinoid. We let $\mathcal{C}_{\Uu}$ be the rigid-analytic space $\operatorname{Spa}\left(\T_{\Uu}\right)$, which is finite over $\Uu$. Clearly the modules $H^i_{\Uu}$ are coherent sheaves on $\mathcal{C}_{\Uu}$. As in Theorem 4.3.3 of \cite{hansen}, and for each $x \in \Uu(\overline{\Q}_p)$, the fibre of $\mathcal{C}_{\Uu}$ over $x$ bijects with the set of generalized eigenspaces for $\T$ appearing in the finite-dimensional vector space $H^i_x = H^i\left(M^\bullet(k_1 + 2\kappa_x, 2-k_2, w)\right)$ for some $i$. 

   From the base-change compatibility of $M^\bullet$, we obtain a second-quadrant Tor spectral sequence for each $x$, which (since $\Oo(\Uu)$ is a Dedekind domain) can be written in the form of exact sequences
   \[ 0 \to H^i_{\Uu} / \mathfrak{m}_x \to H^i_x \to H^{i+1}_{\Uu}[\mathfrak{m}_x] \to 0\]
   for each $i$. If we take $x$ to be the trivial character, and localize at the maximal ideal $\mathfrak{m}_y$, where $y$ is the point of $\mathcal{C}_\Uu$ above $x$ corresponding to the Hecke eigensystem of $\Pi$, then the localization $(H^i_x)_{\mathfrak{m}_y}$ is non-zero only for $i = 1$ (since the control theorem identifies $H^i_x$ with the $U_{\p_1}$-ordinary part of $H^i(X_0(\p_1)_{\overline{\Q}_p}, \myom^{(k_1, 2-k_2, w)})$, and the $\Pi$-generalized eigenspace in this is zero for $i \ne 1$). By Nakayama's lemma we deduce that the localization $(H^i_{\Uu})_{\mathfrak{m}_y}$ vanishes for $i \ne 1$, and for $i = 1$, it is torsion-free and hence free over $\Oo_{\Uu, x}$.

   We now apply the argument (due originally to Chenevier) described in \cite[Proposition 6.6.4]{bergdall-hansen}, with $H^1_{\Uu}$ in place of the module $\mathscr{M}^\epsilon$ of \emph{op.cit.}. Using the fact that the $\Pi$-eigenspace in $H^1_x$ has dimension 1, this argument shows that, after possibly shrinking $\Uu$, the connected component of $\mathcal{C}_\Uu$ containing $y$ maps isomorphically to its image in $\Uu$, and the restriction of the coherent sheaf $H^1_{\Uu}$ to this component is free of rank 1, from which the assertions of the theorem clearly follow.
  \end{proof}

  \begin{rmk}
   In more geometric language, for each $k_2 \ge 2$, there is a ``partial eigenvariety'' parametrising $\p_1$-adic families of weight $(k_1 + 2\kappa, k_2, w)$; and the map from the partial eigenvariety to $\Ww$ is \'etale at the point corresponding to $\Pi$, so we may write down a section over a small enough neighbourhood of its image. (See forthcoming work of C.-Y. Hsu for a more systematic investigation of partial eigenvarieties for Hilbert modular forms.) Note that this step would fail if $k_2 = 1$, since in this case $\Pi$ contributes to both $H^0$ and $H^1$ of $\myom^{(k_1, 1, w)}$; we do not expect partial eigenvarieties for partial weight 1 forms to be \'etale (or flat) over the weight space.
  \end{rmk}

  %%%%%%%%%%%%%%%%%%%%%%
\section{Pushforwards}
 %%%%%%%%%%%%%%%%%%%%%%

 We now consider relations between the higher Hida theory spaces for the group $G = \res_{F/\Q} \GL_2$ developed in the previous section, and spaces of $p$-adic automorphic forms for its subgroup $H = \GL_2$. To avoid confusion between the two groups, we will use subscripts $(-)_G$ or $(-)_H$ as appropriate, so $X_{G, \sph}$ denotes the Hilbert modular variety which we previously denoted simply by $X_\sph$. We fix a choice of neat open compact subgroup $K^{(p)}_G \subset \GL_2(\AFf^{(p)})$ as above, and let $K^{(p)}_H = K^{(p)}_G \cap \GL_2(\Af^{(p)})$; these will be the prime-to-$p$ parts of the levels of our Shimura varieties.

 \subsection{Classical pushforward map}

  We have the canonical inclusion $H\subset G$, inducing a finite map of $\Zp$-schemes
  \[ \iota: X_{H, \sph} \to X_{G, \sph}, \]
  where the Shimura varieties are taken to be of spherical level at $p$. This naturally factors through a map 
  \[ \iota' : X_{H, \sph} \to M_{G, \sph}^{\tor}, \]
  given on the open Shimura variety $Y_{H, \sph}$ by mapping an elliptic curve $E$ to the abelian surface $E \otimes_{\Z} \Oo_F$, with its obvious $\Oo_F$-action and $\cc$-polarization (for $\cc$ representing the class of the inverse different ideal). We thus have an isomorphism of abelian schemes $(\iota')^{\star}(\Aa) = \Ee \otimes_{\Z} \Oo_F$ over $Y_{H, \sph}$, where $\Ee$ is the universal elliptic curve over $Y_{H, \sph}$; this extends to an isomorphism of semiabelian varieties over $X_{H, \sph}$ (essentially by construction, cf.~\cite[\S 4.5]{padicLfnct} in the $\operatorname{GSp}_4$ case; this can also be deduced from the universal property of the toroidal compactifications, cf.~\cite[Chapter IV, Theorem 5.7 (5)]{faltingschai}).

  Moreover, from the isomorphism $(\iota')^{\star}(\Aa) = \Ee \otimes_{\Z} \Oo_F$, where $\Ee$ is the semiabelian variety over $X_{H, \sph}$ extending the universal elliptic curve over the open modular curve (cf.~\S 3.3 of \cite{HMS}), we have isomorphisms for any $n_1, n_2, a, b \in \Z$
  \begin{equation}\label{eq:pullback}
   \omega_{\Ee}^{n_1+n_2}\otimes \left(\wedge^{2} \hh_{\Ee}^{1}\right)^{a+b} \to
   (\iota')^{\star}\left(\omega_{\Aa,1}^{n_1}\otimes \omega_{\Aa,2}^{n_2} \otimes 
   \left(\wedge^{2} \hh_{\Aa,1}^{1}\right)^{a}\otimes \left(\wedge^{2} \hh_{\Aa,2}^{1}\right)^{b}\right),
  \end{equation}
  where $\omega_{\Ee}$ denotes the pullback along the unit section of the sheaf of differentials of $\Ee / X_{H, \sph}$ and one can define similarly the sheaf $\wedge^{2} \hh_{\Ee}^{1}$.

  \begin{defi}
   For any $n \in 2\Z$, we let
   \[
    \myom_H^{n}\coloneqq \omega_{\Ee}^{n}\otimes \left(\wedge^{2} \hh_{\Ee}^{1}\right)^{-\tfrac{n}{2}},
   \]
   so the centre of $H(\Af)$ acts on the coherent cohomology of this sheaf via finite-order characters.
  \end{defi}

  \begin{rmk}
   \label{rmk:twistgl2}
   Note that we have a canonical trivialisation of $\wedge^{2} \hh_{\Ee}^{1}$ (since an elliptic curve has a unique principal polarisation). However, this trivialisation induces a twist of the Hecke action: we have an isomorphism of $H(\Af)$-modules
   \[
    H^0(X_{H, \sph}, \omega_{\Ee}^{n}\otimes \left(\wedge^{2} \hh_{\Ee}^{1}\right)^{m})
    \simeq H^0(X_{H, \sph}, \omega_{\Ee}^{n})\otimes |\det(-)|^{-m}
   \]
   where $|\cdot|$ denotes the adele norm character (mapping a uniformizer at $p$ to $1/p$). Since we shall only consider modular forms for $H$ of even weight in this paper, we can use this freedom to make the central character unitary.
  \end{rmk}

  Let $(k_1, k_2, w)$ be integers with $k_1 = k_2 = w \bmod 2$, as above, with $k_1 > k_2 > 0$, and for $\star\in\{H,G\}$, write $D_\star$ for the cuspidal divisor of $X_\star$. Defining $\myom_G^{(\dots)}$ as in \eqref{eq:defiomegaGmodel}, and using \eqref{eq:pullback} and Remark \ref{rmk:twistgl2}, we see that $\iota$ induces a pullback map:
  \[
   \iota^\star: H^1\left(X_{G, \sph}, \myom_G^{(2-k_1, k_2, -w)}(-D_G)\right) \longrightarrow H^1\left(X_{H, \sph}, \myom_H^{2-k_1 + k_2}(-D_H)\right) \otimes |\det|^w.
  \]
  Recall that the dualizing sheaves for $X_{G, \sph}$ and $X_{H, \sph}$ are $\myom_G^{(2,2,0)}(-D_G)$ and $\myom_H^{2}(-D_H)$, respectively. We obtain  dually a pushforward map
  \begin{equation}
   \label{eq:classicalpushforward}
   \iota_\star: H^0\left(X_{H, \sph}, \myom_H^{k_1-k_2}\right) \otimes |\det|^{-w} 
   \longrightarrow \  H^1\left(X_{G, \sph},\myom_G^{(k_1,2-k_2,w)}\right).
  \end{equation}

 \subsection{$p$-adic pushforward map} 

  We now want to define a pushforward map with target the cohomology of the complex $M_{G}^\bullet(k_1, 2-k_2, w)$ of \cref{thm:classsphcomplex}. 

  Since the $\p_1$-torsion of $E \otimes_{\Z} \Oo_F$ is isomorphic to the $p$-torsion of $E$, the preimage under $\iota$ of the locus $\X_{G, \sph}^{1-\ord} \subseteq \X_{G, \sph}$ where the $\p_1$-torsion is ordinary is the ordinary locus $\X_{H, \sph}^{\ord}$. Hence there is a pushforward map 
  \begin{equation}
   \label{eq:padicpushforward}
   \iota_{p,\star}: H^0(\X_{H, \sph}^{\ord},\myom_H^{k_1-k_2}) 
   \to H^1(\X^{1-\ord}_{G, \sph},\myom_G^{(k_1,2-k_2,w)})
  \end{equation}
  which fits into a commutative diagram
  \[
   \begin{tikzcd}
    H^0(X_{H, \sph}, \myom_H^{k_1-k_2}) \arrow{r}{\iota_\star}\arrow{d}{} &
    H^1(X_{G, \sph}, \myom_G^{(k_1,2-k_2,w)})\arrow{d}\\
    H^0(\X_{H, \sph}^{\ord},\myom_H^{k_1-k_2}) \arrow{r}{\iota_{p,\star}} &
    H^1(\X^{1-\ord}_{G, \sph},\myom_G^{(k_1,2-k_2,w)}).
   \end{tikzcd}
  \]
  where the vertical arrows are given by restriction.

 \subsection{Compatibility of the $p$-adic pushforward with image  of deeper level cohomology}

  We will take the image under the above defined pushforward maps of $p$-adic modular forms for $\GL_2/\Q$ coming from classical modular forms of level $\Gamma_0(p^n)$ for some $n\ge 1$. By definition of \eqref{eq:padicpushforward}, we have a commutative diagram:
  \[
   \begin{tikzcd}
    H^0(X_{H,0}(p^n)_{\Qp}, \myom_H^{k_1-k_2}) \arrow{r}\arrow{d}&
    H^1(X_{G,0}(\p_1^n)_{\Qp},\myom_G^{(k_1,2-k_2,w)})\arrow{d}\\
    H^0(\X^{\ord}_{H,\sph,\Qp},\myom_H^{k_1-k_2}) \arrow{r}{\iota_{p,\star}} &
    H^1(\X^{1-\ord}_{G,\sph,\Qp},\myom_G^{(k_1,2-k_2,w)})
   \end{tikzcd}
  \]
  where the top arrow is the level $p^n$-analogue of the classical pushforward \eqref{eq:classicalpushforward}. Combining the above diagram with \eqref{eq:diagrameta1etan}, we find, for $\eta_n$ as in \cref{defi:etan}, the following commutative diagram (with the same coefficient sheaves as the previous diagram, omitted to fit on the page)
  \begin{equation}
   \label{eq:deeperpushforward}
   \begin{tikzcd}[column sep=small]
    H^0(X_{H,0}(p^n)_{\Qp}) \arrow{r}\arrow{d}&
    H^1(X_{G,0}(\p_1^n)_{\Qp})\arrow{d}\arrow[looseness=0.5, out=0, in=100, swap]{drrr}{p^{1-n}\langle-,\eta_n\rangle} & & &\\
    H^0(\X_{H,\sph, \Qp}^{\ord}) \arrow{r}{\iota_{p,\star}} &
    H^1(\X^{1-\ord}_{G,\sph,\Qp})\arrow{r}{e(U_{\p_1})} &
    e(U_{\p_1}) H^1(\X^{1-\m}_{G,\sph, \Qp})\arrow{r}{\simeq} &
    e(U_{\p_1}) H^1(X_{G,\Qp})%
    \arrow{r}{\langle-,\eta_1\rangle } & \Qp
   \end{tikzcd}
  \end{equation}

 \subsection{Functoriality of Igusa towers}
  \label{sect:funcIgusa}

  Consider the $\GL_2$ Igusa tower $\pi_H:\Ig_H(p^\infty) \to \X_{H, \sph}^{\ord}$ parametrising trivialisations of the multiplicative part $\Ee[p^\infty]^\circ$ of the $p$-divisible group of $\Ee$. Since the pullback to $\X_{H, \sph}^{\ord}$ of the $p$-divisible group $\Aa[\p_1^\infty]$ over $\Mm_G^{\tor}$ is identified with $\Ee[p^\infty]$, we have natural maps $\Ig_H(p^n) \xrightarrow{\iota_{n}}\Ig_G(\p_1^n)$ for each $n$.

  \begin{prop}
   The natural morphism $\Ig_H(p^n) \xrightarrow{\iota_{n}}\Ig_G(\p_1^n)$ is finite for every $n\ge 1$. Moreover the following diagram is Cartesian:
   \[
    \begin{tikzcd}
     \Ig_H(p^n) \arrow{r}{\iota_{n}}\arrow[d, "\pi_H"] &
     \Ig_G(\p_1^n)\arrow[d, "\pi_G"] \\
     \X_{H, \sph}^{\ord} \arrow{r}{\iota'} &
     \M^{\tor, 1-\ord}_{G, \sph}.
    \end{tikzcd}
   \]
  \end{prop}

  \begin{proof}
   Similar to \cite[Prop. 4.6, Lemma 4.7]{padicLfnct}.
  \end{proof}

  Thanks to the above proposition, we obtain
  \begin{equation}\label{eq:pullbackigusa}
   (\iota')^\star(\pi_{G,\star}\Oo_{\Ig_G(\p_1^{\infty})})=\pi_{H,\star}\Oo_{\Ig_{H}(p^{\infty})}
  \end{equation}
  compatibly with the action of $\Zp^\times$ on both sides. There is a similar picture for the dual torsors $\Ig_H^{\vee}(p^\infty)$ and $\Ig_G^{\vee}(\p_1^\infty)$, and we have a canonical isomorphism between $\Ig_H(p^\infty)$ and $\Ig_H^{\vee}(p^\infty)$ given by the polarisation of the universal elliptic curve (cf.~\cref{rmk:twistgl2}).

  Proceeding for $H$ as we did for $G$, we can define a sheaf of $\Lambda \otimes \Oo_{\X^\m_H}$-modules $\Omega_H^{2\kappa}$ interpolating the sheaves $\myom_H^{2a}$ for all $a \in \Z$. The space $H^0(\X_{H, \sph}^{\m}, \Omega_H^{2\kappa})$ is the space of $\Lambda$-adic modular forms of weight $2\kappa$, denoted $M_H(2\kappa)$ in the introduction. From \eqref{eq:pullbackigusa} we obtain pushforward maps of $\Lambda$-modules
  \begin{equation}
   \label{eq:Igusapushforward}
   \iota_{p,\star}^\kappa: H^0(\X_H^{\m}, \Omega_{H}^{k_1-k_2 + 2\kappa}) \to H^1(\X_G^{1-\m},\Omega_G^{(k_1 + 2\kappa,2-k_2,w)}),
  \end{equation}
  interpolating the map of \eqref{eq:padicpushforward} in weight $(k_1 + 2a, k_2, w)$ for any $a \in \Z$ (and likewise for cuspidal sheaves).

  %%%%%%%%%%%%%%%%%%%%%%%%
\section{Nearly sheaves}
 \label{sec:nearly}
 %%%%%%%%%%%%%%%%%%%%%%%%

 \subsection{Sheaves of nearly holomorphic modular forms}

  \begin{defi}
   Given an even integer $k \ge 0$, we define the following sheaves over $X_{H, \sph}$:
   \begin{itemize}
    \item $\Nn^0_{H} \coloneqq\sym^k(\hh_{\Ee}^{1}) \otimes  \left(\wedge^2\hh_{\Ee}^{1}\right)^{-\tfrac{k}{2}}$;
    \item $\Nn^1_{H} \coloneqq \sym^k(\hh_{\Ee}^{1}) \otimes \omega_{\Ee}^2(-D_H)\otimes   \left(\wedge^2\hh_{\Ee}^{1}\right)^{-1-\tfrac{k}{2}}$.
   \end{itemize}
  \end{defi}

  Note that these sheaves are Serre dual to each other, and the action of the centre of $\GL_2$ is trivial on both (see Remark \ref{rmk:twistgl2} for the reason of the choice of the twist). Since we can identify $\omega_\Ee^k$ with the top graded piece in the Hodge filtration of $\sym^k(\hh^1_{\Ee})$, and dually $\omega_\Ee^{-k}$ is identified with the bottom graded piece, we have maps
  \begin{equation}\label{eq:mapsnearlyclassicalH}
   \myom^k_H \hookrightarrow \Nn^0_H, \ \ \ \ \Nn^1_H \twoheadrightarrow \myom_H^{2-k}(-D_H),
  \end{equation}
  Note that these maps are transposes of each other with respect to Serre duality.

  \begin{defi} We write $M_k^{\mathrm{nh}}$ for $H^0(X_H, \Nn^0_H)$, the space of \emph{nearly-holomorphic modular forms} of weight $k$.\end{defi}

  As described in \cite{urban} for example, the map $\myom^k_H \hookrightarrow \Nn^0_{H}$ has a $C^\infty$ section over $\R$, the \emph{Hodge splitting} (which identifies $H^0(X_H, \Nn^0_H)$ with a subspace of the real-analytic sections of $\myom^k$; the image of this map is the space of nearly-holomorphic modular forms in the sense of Shimura). If the level group at $p$ is of the form $K_0(p^n)$, it also has a $p$-adic section over the multiplicative locus $\X_H^{\mathrm{m}}$, the \emph{unit-root splitting}.

 \subsection{Sheaves of (partially) nearly holomorphic Hilbert modular forms}

  We want to proceed analogously for the Hilbert modular surface. Recall that over the compactified moduli space $\Mm_{G, \sph}^{\tor}$ we can consider the sheaf $\hh^1_{\Aa} = \hh^1_{\Aa,1}\oplus \hh^1_{\Aa,2}$.

  Now fix $(k_1, k_2, w)$ with $k_i \ge 0$ and $k_1 = k_2 = w \bmod 2$, and set
  \[ 
   \wedge^2\hh_{0,1}\coloneqq \left((\wedge^2\hh^1_{\Aa,1})^{\tfrac{w-k_1}{2}}\right)
   \otimes \left((\wedge^2\hh^1_{\Aa,2})^{\tfrac{w-2+k_2}{2}}\right), 
   \quad
   \wedge^2\hh_{1,0}\coloneqq \left((\wedge^2\hh^1_{\Aa,1})^{\tfrac{k_1 + 2-w}{2}}\right)
   \otimes \left((\wedge^2\hh^1_{\Aa,2})^{\tfrac{-k_2-w}{2}}\right).
  \]

  \begin{defi}
   We define the following sheaves over $X_{G, \sph}$:
   \begin{itemize}
    \item $\Nn^{0,1}_G\coloneqq (\sym^{k_1} \hh_{\Aa,1})\otimes \omega_{\Aa,2}^{2-k_2}\otimes \wedge^2\hh_{0,1} $;
    \item $\Nn^{1,0}_G\coloneqq (\sym^{k_1} \hh_{\Aa,1} \otimes \omega_{\Aa, 1}^{2}) \otimes \omega_{\Aa, 2}^{k_2} \wedge^2\hh_{1,0}(-D_G)$.
   \end{itemize}
   (Note that these sheaves do indeed descend to $X_{G, \sph}$, since the centre acts as $\nm^w$ on $\Nn^{0, 1}_G$ and as $\nm^{-w}$ on $\Nn^{1,0}_{G}$.) We call an element in $H^1(X_{G, \sph}, \Nn^{0,1}_{G})$ a Hilbert modular form nearly holomorphic in the first variable and anti-holomorphic in the second variable. As before, the Serre dual of this group is  $H^1(X_{G, \sph}, \Nn^{1,0}_G)$.
  \end{defi}

  As in \eqref{eq:mapsnearlyclassicalH}, we obtain natural maps
  \begin{equation}
   \label{eq:mapsnearlyclassicalG}
   \myom_G^{(k_1,2-k_2,w)}\hookrightarrow \Nn^{0,1}_{G},\qquad
   \Nn^{1, 0}_{G}\twoheadrightarrow \myom_G^{(2-k_1, k_2, -w)}(-D_G)
  \end{equation}
  which are Serre duals of each other.

 \subsection{Nearly pushforward maps and unit root splittings} 

  Finally, we want to extend $\iota_{\star}$ to the sheaves considered in \S\ref{sec:nearly}. From the construction we obtain a map of sheaves over $X_{H, \sph}$
  \[
   \iota^{\star}(\Nn_G^{1,0}) \to \Nn_H^1,
  \]
  yielding a pullback map $\iota^*_{\text{nearly}} : H^1(X_{G, \sph}, \Nn^{1,0}_{G}) \to H^1(X_{H, \sph}, \Nn^1_H)$. Dually we find
  \[
   \iota_{\star}^{\text{nearly}}: H^0(X_{H, \sph},\Nn_{H}^0) \to H^1(X_{G, \sph}, \Nn_{G}^{0,1}).
  \]

  Using the relation $\iota^{-1}(\X_{G, \sph}^{1-\ord}) = X_{H, \sph}^{\ord}$, as in the previous section, we obtain a map:
  \[%
   \iota_{p,\star}^{\text{nearly}}: H^0(\X_{H, \sph}^{\ord}, \Nn_{H}^0) \to H^1(\X_{G, \sph}^{1-\ord}, \Nn_{G}^{0,1}),
  \]%
  which is compatible with the classical pushforward map (not only for $X_{H, \sph}$ and $X_{G, \sph}$, but also for higher $p$-power levels as in \eqref{eq:deeperpushforward}).

  \begin{prop}
   We have $p$-adic unit root splittings $u_H,u_G$ of \eqref{eq:mapsnearlyclassicalH},\eqref{eq:mapsnearlyclassicalG} making the following diagram commutative
   \[
    \begin{tikzcd}
     H^0(\X_{H, \sph}^{\ord}, \Nn_H^0) 
     \arrow{r}{\iota_{p,\star}^{\text{nearly}}}\arrow{d}{u_H} &
     H^1(\X_{G, \sph}^{1-\ord}, \Nn_{G}^{0,1})\arrow{d}{u_G}\\
     H^0(\X_{H, \sph}^{\ord},\myom_H^{k_1-k_2})  
     \arrow{r}{\iota_{p,\star}} &
     H^1(\X_{G, \sph}^{1-\ord},\myom_G^{(k_1,2-k_2,w)})
    \end{tikzcd}
   \]
  \end{prop}

  \begin{proof}
   It suffices to show that the unit-root splittings $u_H$ and $u_G$ are compatible via $\iota^*$. The unit-root splitting $u_H$ follows from the fact that the unit-root subspace for the Frobenius map on the fibres is a canonical complement of $\omega_\Ee$ inside $\hh^1_{\Ee}$. The splitting $u_G$ has a similar description using the action of Frobenius on $\hh^1_{\Aa, 1}$; so it suffices to note that the isomorphism $\iota^*\left(\hh^1_{\Aa, 1}\right) \cong \hh^1_{\Ee}$ is compatible with the Frobenii.
  \end{proof}

\section{Coherent classes attached to cusp forms}
 \label{sect:coherent}

 We now change tack slightly, and introduce a specific cup-product in coherent cohomology which is related to $L$-values. For the present section $p$ plays no role (but we shall later use the theory of the previous sections to interpolate these cup-products $p$-adically).

 Throughout this section $\Pi$ is a cuspidal automorphic representation of weight $(k_1, k_2)$, where $k_i \ge 1$ and $k_1\equiv k_2 \bmod 2$, and conductor $\mathfrak{N}$. Thus $\Pif$ has 1-dimensional invariants under the subgroup $K_1(\mathfrak{N}) = \{ g : g = \stbt{\star}{\star}{0}{1} \bmod \mathfrak{N}\}$. For each place $v$ of $\Q$ we write $\Pi_v = \bigotimes_{w \mid v} \Pi_w$.

 \subsection{Whittaker models}

  \begin{notation}[Additive characters]\label{not:addchar} \
   \begin{enumerate}[(i)]
    \item $\psi$ denotes the additive character of $\A_{\Q} / \Q$ satisfying $\psi_\infty(x_\infty) = \exp(-2\pi i x_\infty)$ for $x_\infty \in \R$.
    \item $\psi_F$ denotes the additive character of $\A_F / F$ given by $\psi_F(y) = \psi(\tr_{F / \Q}(\tfrac{y}{\sqrt{D}}))$, where $\sqrt{D}$ denotes the unique square root in $F$ of the discriminant $D = \operatorname{disc}_F$ such that $\sigma_1(\sqrt{D}) > 0$.
   \end{enumerate}
  \end{notation}

  Hence $\psi_F$ is trivial on $\A_\Q \subset \A_F$, and its conductor is 1 (i.e.~it is invariant under $\widehat{\Oo}_F$, and not under any larger fractional ideal). Since any cuspidal automorphic representation of $\GL_n$ is generic, every $\phi \in \Pi$ has a Whittaker expansion
  \begin{equation}
   \label{eq:whittakerfunc}
   \phi(g) = \sum_{\alpha \in F^\times} W_{\phi}\left(\stbt \alpha 0 0 1 g\right), \qquad W_{\phi}(g) = \int_{F \backslash \A_F} \phi(\stbt{1}{x}{0}{1}g) \psi_F(-x)\, \mathrm{d}x
  \end{equation}
  where $\mathrm{d}x$ is the Haar measure on $\A_F$ giving $F \backslash \A_F$ volume 1. The space of functions $W_{\phi}$ on $\GL_2(\A_F)$ is the \emph{Whittaker model} $\Ww(\Pi)$ of $\Pi$; this factors as $\Ww(\Pi_\infty) \otimes \Ww(\Pif)$, where $\Ww(\Pif)$ is the restricted tensor product $\sideset{}{'}\bigotimes_v \Ww(\Pi_v)$ of local Whittaker models at the finite places.

  \begin{defi}\label{def:whittakerinf} \
   \begin{enumerate}[(i)]
    \item  We define the \emph{normalised Whittaker function at $\infty$} as the function on $\GL_2(F \otimes \R)$ which is supported on the identity component and satisfies
    \[ W_\infty^{\mathrm{ah}, 1}\left( \stbt y x 0 1 \stbt{t \cos \theta}{t\sin \theta}{-t\sin \theta} {t\cos \theta}\right) = \sgn(t)^{\kk}\cdot y^{\kk/2} \cdot e^{i (-k_1 \theta_1 + k_2 \theta_2)} \cdot e^{\tfrac{2\pi i }{\sqrt{D}}( (-x_1 + iy_1) + (x_2 + iy_2))}; \]
    for $x, y, t, \theta \in F \otimes \R$ with $y \gg 0$. This is a basis of the minimal $K$-type subspace in the Whittaker model of the $\sigma_1$-antiholomorphic, $\sigma_2$-holomorphic part of $\Pi_\infty$, where $K$ denotes the maximal connected compact subgroup $\operatorname{SO}_2(F \otimes \R)$ of $\GL_2(F \otimes \R)$. 
    \item Casselman's theory of the new vector shows that the 1-dimensional space of $K_1(\mathfrak{N})$-invariants in the Whittaker model of $\Pif$ has a unique basis $W_\f^{\new}$ with $W_\f^{\new}(1) = 1$.
    \item We let $\phi_{\new, \Pi}^{\mathrm{ah}, 1}$ be the unique function in $\Pi$ whose Whittaker function factors as $W_\f^{\new} W_\infty^{\mathrm{ah}, 1}$.
   \end{enumerate}
  \end{defi}

  (If we identify Hilbert modular forms with functions on $\GL_2(\AFf) \times \mathcal{H}_F$ where $\mathcal{H}_F$ is a product of upper half-planes, this $\phi_{\new, \Pi}^{\mathrm{ah}, 1}$ corresponds to the form $\mathcal{F}^{\mathrm{ah}, 1}$ defined in \cite[Lemma 5.2.1]{HMS}.)

 \subsection{Periods}
  \label{sect:periods}

  We choose $w$ with $k_1 \equiv k_2 \equiv w \bmod 2$, and a number field $L$ such that $\Pi_\f \otimes \|\det\|^{w/2}$ is definable over $L$. Then the Whittaker model $\Ww(\Pif \otimes \|\det\|^{w/2})$ is naturally the base-extension to $\C$ of an $L$-linear representation (consisting of the $L$-linear span of the translates of the normalised new vector). We write $\Ww(\Pif)_L$ for the image of this subspace in $\Ww(\Pif)$ (which is independent of $w$). If $X_{G, \sph}$ denotes the Hilbert modular variety of level $K_1(\mathfrak{N})$, the $L$-vector space
  \[ H^1\big(X_{G, \sph, L}, \myom_G^{(2-k_1,k_2,-w)}\big)[\Pif] \]
  is 1-dimensional; and we may choose a basis $\nu_\Pi$ of this space. Note that the choice of $\nu_\Pi$ determines an isomorphism of $\GL_2(\AFf)$-representations
  \begin{equation}
   \label{eq:whittakermap}
   \Ww(\Pif)_{L} \otimes \|\cdot\|^{w/2} \to \varinjlim_{K} H^1(X_G(K)_\C, \myom^{(2-k_1, k_2, -w)})[\Pif],
  \end{equation}
  characterised by sending $W^{\new}_\f$ to $\nu_\Pi$.

  After tensoring with $\C$, there is a canonical isomorphism between $H^1\left(X_{G, \sph, L}, \myom_G^{(2-k_1,k_2,-w)}\right)$ and the space of Hilbert modular forms of weight $(k_1, k_2)$ and level $K_1(\mathfrak{N})$ which are anti-holomorphic at $\sigma_1$ and holomorphic at $\sigma_2$. By comparing Hecke eigenvalues, the image of $\nu_\Pi$ must be a scalar multiple of $\phi_{\new,\Pi}^{\mathrm{ah}, 1}$.

  \begin{defi}
   We define the period $\Omega_\infty(\Pi) \in \C^\times$ (the \emph{occult period} for $\Pi$ at $\sigma_1$ as in \cite{harrishilbert}) to be the scalar such that $\phi_{\new,\Pi}^{\mathrm{ah}, 1} = \Omega_\infty(\Pi) \cdot \mathcal\nu_\Pi$. (This is uniquely determined once $\nu_\Pi$ is chosen; its class in $\C^\times / L^\times$ does not depend on $\nu_\Pi$.)
  \end{defi}

 \subsection{Siegel sections and adelic Eisenstein series}

  We now recall a very general construction of adelic Eisenstein series (as for example in \cite{jacquet}). Let $\Ss(\A^2, \C)$ denote the space of Schwartz functions $\A^2 \to \C$. For $\Phi \in \Ss(\A^2, \C)$, and $\chi$ a Dirichlet character, we define a global Siegel section and associated Eisenstein series by
  \begin{subequations}
   \begin{align}
    \label{eq:siegelsec}
    f^{\Phi}(g ; \chi, s)\coloneqq\|\operatorname{det} g\|^s \int_{\A^{\times}} \Phi((0, a) g) \hat\chi(a)\|a\|^{2 s}\,\mathrm{d}^{\times} a, \\
    E^{\Phi}(g ; \chi, s)\coloneqq\sum_{\gamma\in B(\Q)\backslash \GL_2(\Q)} f^{\Phi}(\gamma g; \chi, s).
   \end{align}
  \end{subequations}
  Here $\hat{\chi}$ is the character of $\Q^{\times}\backslash \A^{\times}/\R^{\times}_{>0}$ corresponding to $\chi$ as in \cite[\S2.2]{padicLfnct}. The sum converges absolutely and uniformly on any compact subset of $\{s: \re(s)>1\}$, it has a meromorphic continuation in $s$, and defines a function on the quotient $\GL_2(\Q) \backslash \GL_2(\A)$ that transforms under the center by $\hat\chi^{-1}$.

  We now make a choice of the Schwartz function at $\infty$ which recovers familiar holomorphic or nearly-holomorphic Eisenstein series. For $k\in \Z_{\ge 1}$, let $\Phi_{\infty}^{(k)} \in \Ss(\R^2, \C)$ be the function given by
  \[
   \Phi_{\infty}^{(k)}(x, y) \coloneqq 2^{1-k}(x+i y)^k e^{-\pi\left(x^2+y^2\right)}.
  \]
  Then, for $\Phi=\Phi_\f\cdot \Phi_{\infty}^{(k)}$, where $\Phi_\f \in \Ss(\Af^2, \C)$ is arbitrary, and $1-\tfrac{k}{2} \le s \le \tfrac{k}{2}$ with $s = \tfrac{k}{2} \bmod \Z$, the function $E^{k, \Phi_\f}(-; \chi, s) \coloneqq E^{\Phi_\f \cdot \Phi_{\infty}^{(k)}}(-; \chi, s)$ is a nearly-holomorphic modular form. Moreover, if $\Phi_\f$ and $\chi$ take values in a number field $L$, then this form is defined over $L$ as a coherent cohomology class.

 \subsection{The global zeta integral}

  \begin{defi}\label{def:globalzeta}
   Let $\nu$ be a Dirichlet character. For $\Pi$ as in \cref{sect:coherent}, $\phi \in \Pi$, and $\Phi \in \Ss(\A^2, \C)$, we define
   \[ Z(\phi, \Phi; \nu, s) = \int_{Z_H(\A) H(\Q) \backslash H(\A)} \phi(h) \hat\nu(\det h)^{-1} E^{\Phi}(h; \chi \nu^{-2}, s)\, \mathrm{d}h,\]
   where $\mathrm{d}h$ is the Haar measure on $H(\A)$ (normalised as in \cite[Remark 1.6]{harriskudla92}), and $\hat\nu$ the adelic character associated to $\nu$.
  \end{defi}

  Let $W_{\phi}$ be the Whittaker function of $\phi$ as in \eqref{eq:whittakerfunc}. %
  We suppose $W_\phi = \prod_v W_v$ and $\Phi = \prod_v \Phi_v$ are pure tensors. Then the global integral factors as $Z(\phi, \Phi; \nu, s) = \prod_v Z(W_v, \Phi_v; \hat\nu_v, s)$, where
  \[ Z(W_v, \Phi_v; \hat\nu_v, s) \coloneqq \int_{(Z_H N_H \backslash H)(\Q_v)} W_v(h_v) \hat\nu_v(\det h_v)^{-1} f^{\Phi_v}(h_v; \hat\chi_v \hat\nu_v^{-2}, s)\, \mathrm{d}h_v, \]
  where $f^{\Phi_v}(\dots)$ is the local Siegel section attached to $\Phi_v$ (defined analogously to \eqref{eq:siegelsec}; see \cite[\S 8.1]{padicLfnct}).
  \begin{prop}
   \label{prop:localAsai}
   Suppose $\ell$ is a finite place. For each $W_\ell$ and $\Phi_\ell$, there exists a polynomial $C_\ell(W_\ell, \Phi_\ell; X) \in \C[X, X^{-1}]$ such that for all $\nu$ unramified at $\ell$, we have an equality of meromorphic functions of $s$
   \[ Z(W_\ell, \Phi_\ell; \hat\nu_\ell, s) = L_{\As}(\Pi_\ell \times \hat\nu_\ell^{-1}, s) \cdot C_\ell\left(W_\ell, \Phi_\ell; \nu(\ell)^{-1} \ell^{-s}\right),\]
   where $L_{\As}(\Pi_\ell \times \hat\nu_\ell^{-1}, s)$ denotes the local Asai $L$-factor (as introduced in \cite[Appendix]{flicker93}).
   Moreover,
   \begin{enumerate}[(i)]
    \item for any $\ell$, the ideal of $\C[X, X^{-1}]$ generated by the $C_\ell\left(W_\ell, \Phi_\ell; X\right)$ is the unit ideal;
    \item if $\Pi_\ell$ is unramified and $W_\ell, \Phi_\ell$ are the normalised spherical data, then $C_\ell\left(W_\ell, \Phi_\ell; X\right) = 1$.
   \end{enumerate}
  \end{prop}

  \begin{proof}
   Since the character $\hat\nu_\ell$ of $\Q_\ell^\times$ is unramified, we may reduce to the case $\hat\nu_\ell = 1$ by replacing $s$ with $s + \alpha$ for any $\alpha \in \C$ such that $\ell^\alpha = \nu(\ell)$. Then our integral $Z(W_\ell, \Phi_\ell; \hat\nu_\ell, s)$ coincides with Flicker's Asai zeta integral $\Psi(s, W_\ell, \Phi_\ell)$ in the notation of \cite[Appendix]{flicker93}. Flicker \emph{defines} the Asai $L$-factor as the unique $L$-factor generating the fractional ideal of $\C[\ell^{\pm s}]$ formed by the values of this integral as $\Phi_\ell$ and $W_\ell$ vary, so the existence of the polynomials $C_\ell$, and the fact that they generate the unit ideal, are true by definition. It remains to check that this definition of the Asai $L$-factor coincides with the Langlands--Deligne $L$-factor, which is due to Matringe (see \cite[Remark 2.14]{loefflerwilliams18}). The unramified computation (ii) is carried out for $\GL_n$ in \cite[\S 3]{flicker88}; see Lemmas 2.6 and 3.13 of \cite{grossi19} for an alternative approach in our $\GL_2$ setting.
  \end{proof}

\section{\texorpdfstring{$p$-adic Eisenstein measure and interpolation of period integrals}{p-adic Eisenstein measure and interpolation of period integrals}}

 \subsection{\texorpdfstring{$\p_1$-stabilisation}{p1-stabilisation}}

  We now assume $\Pi$ is ordinary at $\p_1$, and we choose a finite extension $E / \Qp$ containing the field of definition $L$ of $\Pi$ (via our isomorphism $\C \cong \overline{\Q}_p$). Let $\alpha_1, \beta_1$ be the parameters associated to $\Pi_{\p_1}$ as in \cref{not:alphabeta}, ordered so that $v_p(\alpha_1)=0$.

  There is a natural pullback map
  \[
   {\rm pr}^\star: H^1\left(X_{G,\sph, E}, \myom_G^{(2-k_1,k_2, -w)}\right)
   \to H^1(X_{G,0}(\p_1)_E, \myom_G^{(2-k_1,k_2, -w)})
  \]
  We let
  \begin{equation}
   \label{eq:definup}
   \nu_{\Pi,\p_1}\coloneqq \left(1-\frac{\beta_1}{U_{\p_1}'}\right) \pr^\star(\nu_\Pi).
  \end{equation}
  This makes sense because the characteristic polynomial of $U_{\p_1}'$ acting on the Iwahori-invariants of $\Pi_{\p_1}$ is $(X - \alpha_1)(X-\beta_1)$, and $\alpha_1\beta_1 \ne 0$, so $U_{\p_1}'$ is invertible on this space. 
  By definition $\nu_{\Pi,\p_1}$ lies in \\
  $e(U_{\p_1}') H^1\left(X_{G,0}(\p_1)_E, \myom_G^{(2-k_1,k_2, -w)}(-D_G)\right)[\Pif]$; and it corresponds via \eqref{eq:whittakermap} to the basis vector $W_{\p_1, \alpha_1} \cdot W_{\f}^{(\p_1), \new}$ of the ordinary $U_{\p_1}'$-eigenspace in $\Ww(\Pif)^{K_{\p_1, 1} \cdot K_1(\mathfrak{N})^{(\p_1)}}$, where $W_{\f}^{(\p_1), \new}$ is the new vector of $\Ww(\Pif^{(\p_1)})$ (normalised to have value 1 at the identity), and 
  \[ W_{\p_1, \alpha_1} \coloneqq \left(1-\frac{\beta_1}{U_{\p_1}'}\right) W_{\p_1, \sph} \in \Ww(\Pi_{\p_1}),\]
  where $W_{\p_1, \sph}$ is the normalised spherical vector of $\Ww(\Pi_{\p_1})$. (Here we normalise the action of $U_{\p_1}'$ on the Whittaker model to be compatible with \eqref{eq:whittakermap}.)

  We shall need to extend this to a family of eigenvectors with deeper levels at $\p_1$ (compare \cite[\S 5.8]{padicLfnct} in the $\operatorname{GSp}_4$ case). For $n \ge 1$, we let $W_{\p_1, \alpha_1}[n]$ denote the unique $U_{\p_1}' = \alpha_1$ eigenvector in the $K_{\p_1, n}$-invariants of $\Ww(\Pi_{\p_1})$ whose image under the idempotent $\frac{1}{p^{n-1}} \sum_{\gamma \in K_{\p_1, 1} / K_{\p_1, n}} \gamma$ is $W_{\p_1, \alpha_1}$. Then the $W_{\p_1, \alpha_1}[n]$ are compatible under the normalised trace maps.

 \subsection{Katz $p$-adic Eisenstein measure}

  Let us now fix the choice of a Schwartz function $\Phi^{(p)}\in \Ss((\Af^{(p)})^2,\C)$ at places \emph{away from $p$}, and $\chi^{(p)}$ a Dirichlet character of conductor coprime to $p$ and such that
  $\stbt{a}{0}{0}{a} \cdot \Phi^{(p)}=\hat{\chi}^{(p)}(a)^{-1} \Phi^{(p)}$ for $a \in(\hat{\Z}^{(p)})^{\times}$.

  Now let us consider a pair $\mu, \nu$ of Dirichlet characters of $p$-power conductor. We will now attach to such a pair an element $\Phi_{p,\mu,\nu}\in \Ss(\Qp^2,\C)$. First recall that if $v$ is any place of $\Q$, $\phi\in \Ss(\Q_v,\C)$ is a Schwartz function in one variable, and $\psi_v$ is our fixed choice of additive character on $\Q_v$, the usual Fourier transform $\hat{\phi}$ of $\phi$ is given by
  \[ \widehat{\phi}(x)=\int_{\Q_v}\phi(y)\psi_v(yx)dy. \]
  For any character $\xi:\Qp^{\times}\to \C^{\times}$, let $\phi_\xi$ be the Schwartz function defined by $\phi_\xi(x)\coloneqq \ch_{\Zp^{\times}}(x)\xi(x)$. We let
  \[
   \Phi_{p,\mu,\nu} (x,y)= \phi_{\mu}(x)\widehat{\phi}_{\nu}(y).
  \]
  We regard this as taking values in $\overline{\Q}_p$ via our fixed choice of isomorphism from $\C$.

  \begin{thm}\label{thm:eisensteinmeasure}
   Let $\Lambda'_E = \Lambda_E \mathop{\hat{\otimes}} \Lambda_E$, where $E$ is a $p$-adic field containing $L$, and let $(\kappa_1,\kappa_2)$ be the two canonical characters into $\Lambda'_E$. Then for each $\Phi^{(p)}$ taking values in $L$, there exists a two-variable measure 
   \[ \Ee^{\Phi^{(p)}}\left(\kappa_1,\kappa_2 ; \chi^{(p)}\right)\in H^0(\X_H^{\m},\myom_H^{\kappa}(-D_H))\mathop{\hat{\otimes}}_{\Lambda, \kappa_1 + \kappa_2 + 1} \Lambda'_E \]
   such that its specialisation at $(a+\mu, b+ \nu)$ for any $a,b\ge 0$ is the $p$-adic modular form associated to the algebraic nearly-holomorphic modular form
   \[
    g \mapsto \widehat{\nu}(\det g)^{-1} \cdot 
    E^{\left(a+b+1, \Phi^{(p)}\Phi_{p, \mu, \nu}\right)}\left(g; \chi^{(p)}  \mu\nu^{-1}, \frac{b-a+1}{2}\right) 
    \in M_{a+b+1}^{\mathrm{nh}}.
   \]

   Moreover, let $\Lambda_E =\Zp[[\Zp^{\times} ]]\otimes E$ with canonical character $\sigma$, and let $t \in \Z_{\ge 2}$ be even. Then for each $\Phi^{(p)}$ taking values in $L$, there exists an element
   \[
    \Ee_t^{\Phi^{(p)}}\left(\sigma ; \chi^{(p)}\right)\in H^0(\X_H^{\m},\myom_H^t(-D_H)) \mathop{\hat{\otimes}}_{\Zp} \Lambda_{E}
   \]
   such that its specialisation at $s + \nu$ for any integer $s$ with $1 - \tfrac{t}{2} \le s \le \tfrac{t}{2}$ is the $p$-adic modular form associated to the algebraic nearly-holomorphic modular form
   \[
    g \mapsto \widehat{\nu}(\det g)^{-1} \cdot E^{\left(t, \Phi^{(p)}\Phi_{p, \nu^{-1}, \nu}\right)}\left(g; \chi^{(p)}  \nu^{-2}, s\right) \in M_{t}^{\mathrm{nh}}.
   \]
  \end{thm}

  \begin{proof} The first statement is just a consequence of \cite[Theorem 7.6]{padicLfnct}, which in turn is a reformulation of the existence of Katz' Eisenstein measure (originally introduced in \cite{katzpadic}).
   The second statement follows from the first one, as the one-variable object we are looking for is obtained letting
   \[ \Ee^{\Phi^{(p)}}_t(\sigma; \chi^{(p)}) \coloneqq
    \Ee^{\Phi^{(p)}}\left(\tfrac{t}{2} - \sigma, \sigma + \tfrac{t}{2} - 1 ; \chi^{(p)}\right).\qedhere
   \]
  \end{proof}

  \begin{rmk}
   In order to construct the one-variable $p$-adic $L$-function, we are going to consider the measure $\Ee_{k_1-k_2}^{\Phi^{(p)}}\left(\kappa ; \chi^{(p)}\right)$. The interpolation property tells us that, fixing the $p$-power conductor character $\nu$, we can get the Eisenstein series $ \widehat{\nu}(\operatorname{det} g)^{-1} \cdot E^{\left(k_1-k_2, \Phi^{(p)}\Phi_{p, \nu^{-1}, \nu}\right)}(g; \chi^{(p)}  \nu^{-2}, s)$ for values of $s$ in the range
   \[
    1-\tfrac{k_1-k_2}{2}\le s\le \tfrac{(k_1-k_2)}{2}.
   \]
   Note this is exactly the range \eqref{eq:critrangeintro} from the introduction.
  \end{rmk}

 \subsection{One-variable interpolation of zeta-integrals}
  \label{sect:onevarpadic}
  Let $\Pi$ be a cuspidal automorphic representation of $\GL_2(\A_F)$, of weight $(k_1, k_2)$ with $k_1 = k_2 \bmod 2$ and $k_1 > k_2 \ge 1$, and of level prime to $p$ (as in \cref{thmintro:cycloL}).

  We choose the following local data away from $p \infty$:
  \begin{itemize}
   \item a Whittaker function $W_\f^{(p)} \in W(\Pif^{(p)})$ which is defined over $L$;
   \item a Schwartz function $\Phi^{(p)}$ defined over $L$;
   \item a level group $K_G^{(p)}$ such that $K_G^{(p)}$ fixes $W_\f^{(p)}$ and $K^{(p)}_H \coloneqq K_G^{(p)} \cap H$ fixes $\Phi^{(p)}$.
  \end{itemize}

  We let $W_{p, \alpha_1} = W_{\p_1, \alpha_1} \cdot W_{\p_2, \sph} \in \Ww(\Pi_p)$, where $W_{\p_2, \sph}$ is the normalised spherical vector of $\Ww(\Pi_{\p_2})$, and $W_{\p_1, \alpha_1}$ the $U_{\p_1}'$-eigenvector in $\Ww(\Pi_{\p_1})$ chosen above; and we let $W_{\f, \alpha_1} = W_\f^{(p)} \cdot W_{p,\alpha_1} \in \Ww(\Pif)$. We let $\breve{\nu}_{\Pi, \alpha}$ be its image under the map \eqref{eq:whittakermap}. If we define $\breve{\nu}_{\Pi, \alpha}[n]$ for $n \ge 1$ analogously using $W_{\p_1, \alpha_1}[n]$, then the classes $\left(\breve{\nu}_{\Pi, \alpha}[n]\right)_{n \ge 1}$ are a trace-compatible system as in \cref{defi:etan}.

  \begin{defi}
   We define
   \[
    \mathcal{Z}_{\Lambda}\left(W_\f^{(p)}, \Phi^{(p)}; \sigma\right) = \operatorname{Vol}(K^{(p)}_H) \cdot \left\langle \breve{\nu}_{\Pi, \alpha}, \iota_{p, \star}\left(\Ee_{k_1 - k_2}^{\Phi^{(p)}}(\sigma; \chi^{(p)})\right)\right\rangle_{\X_{G, 0}(\p_1)^{1-\m}} \in \Lambda_E.
   \]
  \end{defi}

  Here the cup-product is taken in the cohomology of the Shimura variety of prime-to-$p$ level $K_G^{(p)}$; the volume factor makes the product independent of the choice of level group. 

  \begin{prop}
   For any locally-algebraic character $s + \rho$, with $s$ in the critical range, we have
   \begin{equation}
    \label{eq:globalint}
    \mathcal{Z}_{\Lambda}\left(W_\f^{(p)}, \Phi^{(p)}; s + \rho\right) =
    \frac{(p + 1)}{\Omega_\infty(\Pi)} Z\left(W_{\f, \alpha_1}[n]\cdot W_\infty^{\mathrm{ah}, 1}, \Phi^{(p)} \Phi_{p, \rho^{-1}, \rho} \Phi_\infty^{(k_1 - k_2)}; \hat\rho^{-1}, s\right),
   \end{equation}
   for all $n \gg 0$ (depending on $\rho$), where $Z(\dots)$ is the global zeta integral of \cref{def:globalzeta}.
  \end{prop}

  \begin{proof}

   For $s$ in the critical range, the specialization of the Katz Eisenstein measure $\Ee_{k_1 - k_2}^{\Phi^{(p)}}$ at $s + \rho$ is the image under the unit-root splitting of a classical nearly-holomorphic form $\mathcal{G}_{s, \rho}$, whose level at $\p_1$ is of the form $K_{\p_1, n}$ for some\footnote{It suffices to take $n = \max(1, 2m)$ if $\rho$ has conductor $p^m$, but the exact value does not matter here.} $n \ge 1$.

   Using the compatibility of the higher Hida theory cup-product with classical cup-product pairings at higher level, as in the diagram \eqref{eq:deeperpushforward}, we have the equality
   \begin{align*} 
    \mathcal{Z}_{\Lambda}\left(W_\f^{(p)}, \Phi^{(p)}; s + \rho\right) &=
    p^{1 - n} \operatorname{Vol}(K_H^{(p)}) \cdot \left\langle \breve{\nu}_{\Pi, \alpha}[n], \iota_* (\mathcal{G}_{s, \rho})\right\rangle_{X_{G, 0}(\p_1^n)} \\
    &= p^{1 - n} \operatorname{Vol}(K_H^{(p)}) \cdot \left\langle \iota^*\left(\breve{\nu}_{\Pi, \alpha}[n]\right), \mathcal{G}_{s, \rho}\right\rangle_{X_{H, 0}(p^n)}
   \end{align*}
   using the adjointness of pushforward and pullback.

   We now calculate this cup-product using the isomorphism between coherent cohomology of toroidal compactifications and spaces of harmonic differential forms, following a now-standard technique originally introduced by Harris \cite{harrishilbert}. Under these isomorphisms the cup-product corresponds to integration over $X_{H, 0}(p^n)(\C)$. If we identify this space with a quotient of the adelic symmetric space $[H] = Z_H(\A) H(\Q) \backslash H(\A)$, and use \cref{thm:eisensteinmeasure} to relate $\mathcal{G}_{s, \rho}$ to Jacquet's adelic Eisenstein series, then the resulting integral is exactly $(p + 1) Z(\breve{\phi}_{\Pi, \alpha}, \Phi; \hat\rho^{-1}, s)$, where $\breve{\phi}_{\Pi, \alpha}[n] \in \Pi$ is the image of $\breve{\nu}_{\Pi, \alpha}[n]$ under Harris' comparison map, and $\Phi = \Phi^{(p)} \cdot \Phi_{p, \rho^{-1}, \rho} \cdot \Phi_\infty^{(k_1 - k_2)}$. (The factor $(p + 1)$ arises from a comparison of volumes, since $\operatorname{Vol}(K_{H, 0}(p^n)) = \frac{1}{p^{n-1}(p + 1)}\operatorname{Vol}(K_H^{(p)})$.)

   We now note that the Whittaker function associated to $\breve{\phi}_{\Pi, \alpha}[n]$ is exactly $\frac{1}{\Omega_\infty(\Pi)} W_{\f, \alpha_1}[n]\cdot W_\infty^{\mathrm{ah}, 1}$, by the definition of the Whittaker period $\Omega_\infty(\Pi)$.
  \end{proof}

  \begin{cor}\label{cor:zetaint}
   If $W_{\f}^{(p)}$ factors as a product of local Whittaker functions $W_\ell \in \Ww(\Pi_\ell)$, then for $n \gg 0$ we have
   \begin{multline*}
    \mathcal{Z}_{\Lambda}\left(W_\f^{(p)}, \Phi^{(p)}; s + \rho\right) =(p+1) \frac{L_{\As}(\Pi, \rho^{-1}, s)}{\Omega_\infty(\Pi)} \times \\ 
    \frac{Z(W_{p, \alpha_1}[n], \Phi_{p, \rho^{-1}, \rho}; \hat\rho_p^{-1}, s)}{L_{\As}(\Pi_p, \rho_p^{-1}, s)} \cdot Z(W_\infty; \Phi_\infty^{k_1 - k_2}, s) \cdot \prod_{\ell \ne p} C_\ell(W_\ell, \Phi_\ell; \ell^{-s} \rho(\ell)^{-1}),
   \end{multline*}
   where the local terms $C_\ell$ are as in \cref{prop:localAsai} (and all but finitely many of these are 1).
  \end{cor}

  \begin{proof}
   This follows from \eqref{eq:globalint} by expanding $Z\left(W_{\f, \alpha_1}[n]\cdot W_\infty^{\mathrm{ah}, 1}, \Phi^{(p)} \Phi_{p, \rho^{-1}, \rho} \Phi_\infty^{(k_1 - k_2)}; \rho, s\right)$ as a product of local zeta integrals, and writing the local integral at each finite place as the product of the $L$-factor and a correction term as in \cref{prop:localAsai}.
  \end{proof}

  We shall consider the following two inputs to $\mathcal{Z}_\Lambda$:

  \subsubsection{Primitive test data}

   From part (i) of \cref{prop:localAsai} we know that for each $\ell \ne p$, the ideal of $\C[X, X^{-1}]$ generated by the local terms $C_\ell(\dots)$, as $W_\ell$ and $\Phi_\ell$ vary, is 1. So for each such $\ell$ we may find a collection of local data $W_{\ell, i}, \Phi_{\ell, i}$ and $a_{\ell, i} \in \C[X, X^{-1}]$, indexed by $i$ in some finite set $I_\ell$,  such that $\sum_{i \in I_\ell} a_{\ell, i}(X) C_\ell(W_{\ell, i}, \Phi_{\ell, i}; X) = 1$. 

   We now descend from $\C$ to the coefficient field $L$. From the definition of the local zeta-integrals, one checks easily that if $W_\ell$ is defined over $L$, and $\Phi_\ell$ takes values in $L$, then $C_\ell(W_\ell, \Phi_\ell; X)$ lies in $L[X, X^{-1}]$. (Compare \S 8.3.1 of \cite{padicLfnct}.) The twist by $\|\det\|^{-w/2}$ is harmless since the zeta integral only depends on the restriction of $W_\ell$ to $\GL_2(\Q_\ell) \subset \GL_2(F \otimes \Q_\ell)$, and on this subgroup $\|\det\|^{-w/2}$ takes values in $\Q$. Hence we can choose local test data $(W_{\ell, i}, \Phi_{\ell, i}, a_{\ell, i})_{i \in I_\ell}$ \emph{defined over $L$} such that $\sum_{i \in I_\ell} a_{\ell, i}(X) C_\ell(W_{\ell, i}, \Phi_{\ell, i}; X) = 1$. Of course, for all but finitely many primes we can (and do) take $I_\ell = \{1\}$, $a_{\ell, 1} = 1$ and $W_{\ell, 1}$ and $\Phi_{\ell, 1}$ the normalised spherical data. 

   We let $I = \prod_{\ell \ne p} I_\ell$ (a finite set, since all but finitely many of the $I_\ell$ are singletons); and for $i = (i_\ell)_{\ell} \in I$ we set $W_i^{(p)} = \prod_\ell W_{\ell, i_\ell} \in \Ww(\Pif^{(p)})_L$, and similarly $\Phi_i^{(p)}$, so we have measures $\mathcal{Z}_{\Lambda}(W_i^{(p)}, \Phi_i^{(p)}; \sigma) \in \Lambda_E$. We set $a_i(\sigma) = \prod_\ell a_{\ell, i_\ell}(\ell^{-\sigma}) \in \Lambda_E$, where $\sigma$ is the universal character $\Zp^\times \to \Lambda_E^\times$ as above, so that evaluating $a_i$ at $s + \rho$ gives $\prod_\ell a_{\ell, i_\ell}(\ell^{-s} \rho(\ell)^{-1})$. Hence the element
   \[ \sum_{i \in I} a_i(\sigma) \cdot \mathcal{Z}_{\Lambda}(W_i^{(p)}, \Phi_i^{(p)}; \sigma) \in \Lambda_E \]
   has the property that its value at $s + \rho$, for all critical $s$ and finite-order $\rho$, is given by
   \[ (p+1) \frac{L_{\As}(\Pi \times \rho^{-1}, s)}{\Omega_\infty(\Pi)} \cdot \frac{Z(W_{p, \alpha_1}[n], \Phi_{p, \rho^{-1}, \rho}; \rho_p, s)}{L_{\As}(\Pi_p \times \rho_p^{-1}, s)} \cdot Z(W_\infty; \Phi_\infty^{(k_1 - k_2)}, s) \]
   for $n \gg 0$ depending on $\rho$.

   \begin{defi}\label{defi:onevariable}
    We define $\mathcal{L}_{p, \As}(\Pi)$ by the formula
    \begin{equation}
     \label{eq:sillyfactor}
     \mathcal{L}_{p, \As}(\Pi)(\sigma) = \tfrac{1}{p} \cdot (\sqrt{D})^{1-(k_1 + k_2)/2 - \sigma}(-1)^\sigma \cdot \sum_{i \in I} a_i(\sigma) \cdot \mathcal{Z}_{\Lambda}(W_i^{(p)}, \Phi_i^{(p)}; \sigma) \in \Lambda_E. \end{equation}
   \end{defi}

   We shall evaluate the local factors at $p$ and $\infty$ below, and this will give the interpolation property of \cref{thmintro:cycloL}. The extra terms multiplying the sum in \eqref{eq:sillyfactor} will serve to cancel out various unwanted terms arising from the local zeta-integrals.

  \subsubsection{Imprimitive test data}
   \label{sect:imprimtest}
   Since we do not know how the primitive test data behave as we vary $\Pi$ over the specialisations of a $p$-adic family, we will also work with a less optimal, but explicit, alternative set of data. These ``imprimitive test data'' are the inputs $W_{\f}^{(p)}$ and $\Phi^{(p)}$ given by the following recipe, where $N$ is the integer generating $\mathfrak{N} \cap \Z$.
   \begin{itemize}
    \item $W_{\f}^{(p)}$ is the normalised new-vector,
    \item $\Phi^{(p)} = \prod_{\ell \ne p} \Phi_\ell$, where $\Phi_\ell = \ch(\Z_\ell^2)$ if $\ell \nmid N$, and if $v_\ell(N) = r > 0$, we define
    \[ \Phi_\ell = (\ell^2 - 1) \ell^{2r - 2} \cdot \ch( \ell^r \Z_\ell, 1 + \ell^r \Z_\ell).\]
   \end{itemize}
   Thus we may take $K^{(p)}_G = K_{G, 1}(\mathfrak{N})^{(p)}$. The imprimitive (one-variable) $L$-function $\mathcal{L}^{\imp}_{p, \As}(\Pi)$ is defined as the zeta-integral $\mathcal{Z}_{\Lambda}\left(W_\f^{(p)}, \Phi^{(p)}; \sigma\right)$ for this choice of data, multiplied by the factor $\tfrac{1}{p} \cdot (\sqrt{D})^{1-(k_1 + k_2)/2 - \sigma}(-1)^\sigma$ as before.

 \subsection{Two-variable interpolation}

  We now consider the case of families. Given a $\p_1$-ordinary family $\underline{\Pi}$ of level $\mathfrak{N}$ over some affinoid $\Uu$, as above, we clearly obtain a dual family $\underline{\Pi}^\vee$ by twisting by the inverse of the nebentype, whose specialisations are the duals of the specialisations of $\underline{\Pi}$. So, after shrinking $\Uu$ if necessary, \cref{thm:families} gives a free rank 1 direct summand 
  \[ H^1(k_1 + 2\kappa_{\Uu}, 2-k_2, w)[\underline{\Pi}^\vee]\ \subseteq H^1(M^\bullet_G(k_1 + 2\kappa_\Uu, 2-k_2, w))\]
  associated to $\underline{\Pi}^\vee$. We choose an arbitrary basis vector $\nu_{\underline{\Pi}}$ of the $\Oo(\Uu)$-dual of the free rank one $\Oo(\Uu)$-module $H^1(k_1 + 2\kappa_{\Uu}, 2-k_2, w)[\underline{\Pi}^\vee]$.

  Via Serre duality, we can interpret the specialisation of $\nu_{\underline{\Pi}}$ at an integer weight $a$ as a basis of the $U_{\p_1}'$-ordinary stabilization of the $\Pi[a]$-eigenspace in cohomology over our $p$-adic field $E$, but it may not necessarily descend to a number field. So we choose an algebraic basis $\nu_{\Pi[a]}$ of the prime-to-$p$-level new eigenspace, as in \cref{sect:periods}, defined over the coefficient field $L_a$ of $\Pi[a]$ (which is a subfield of $E$). By \eqref{eq:definup} we obtain a basis $\nu_{\Pi[a], \p_1}$ of the corresponding $U_{\p_1}'$-ordinary eigenspace, which must be a scalar multiple of $(\nu_{\underline{\Pi}})_{|a}$. This defines a \emph{pair} of periods
  \[ \Omega_p(\Pi[a]) \in E^\times \qquad\text{and}\qquad \Omega_\infty(\Pi[a]) \in \C^\times \]
  with
  \[ (\nu_{\underline{\Pi}})_{|a} = \Omega_p(\Pi[a]) \cdot \nu_{\Pi[a], \p_1}, \qquad \phi_{\new, \Pi[a]}^{\mathrm{ah}, 1} = \Omega_\infty(\Pi[a]) \cdot \nu_{\Pi[a]}.\]

  \begin{defi}
   We let the two-variable measure $\mathcal{L}_{p, \As}^{\imp}(\underline{\Pi})\in \Oo(\Uu) \mathop{\hat\otimes} \Lambda_E$ be defined by
   \[
    \mathcal{L}_{p, \As}^{\imp}(\underline{\Pi})(\kappa, \sigma) = (\star) \cdot
    \left\langle \nu_{\underline{\Pi}}, \iota_{p,\star}^{\kappa}\left(\Ee^{\Phi^{(p)}}(\tfrac{k_1-k_2}{2} + \kappa_{\Uu} - \sigma, \sigma + \kappa_{\Uu} + \tfrac{k_1-k_2}{2} - 1; \chi^{(p)})\right)\right\rangle,
   \]
   where $\iota_{p,\star}^{\kappa}$ was defined in \eqref{eq:Igusapushforward}, $(\star)$ is the factor from \eqref{eq:sillyfactor} (with $k_1 + 2\kappa$ in place of $k_1$), $\chi^{(p)}$ is the common central character of all the $\Pi[a]$, and $\Phi^{(p)}$ is determined by $\mathfrak{N}$ as above. Here $\langle \nu_{\underline{\Pi}}, -\rangle$ denotes the composition of projection to $H^1_{\m, \varnothing}(k_1 + 2\kappa_{\Uu}, 2-k_2, w)[\underline{\Pi}^\vee] \otimes_{\Oo(\Uu)} \left(\Oo(\Uu) \mathop{\hat\otimes} \Lambda_E\right)$ and applying the linear functional $\nu_{\underline{\Pi}}$.
  \end{defi}

  At an integer $a \in \Uu \cap \Z_{\ge 0}$, the map $\iota_{p,\star}^{\kappa}$ specialises to the $p$-adic pushforward $\iota_{p, \star}$, and the two-parameter Eisenstein series in the definition of $\mathcal{L}_{p, \As}^{\imp}(\underline{\Pi})$ specialises to the one-parameter Eisenstein series used to define $\mathcal{L}^{\imp}_{p, \As}(\Pi[a])$. On the other hand, $\nu_{\underline{\Pi}}$ specialises to $\Omega_p(\Pi[a]) \cdot \nu_{\Pi[a], \p_1}$. Hence we have the specialisation formula
  \[ \mathcal{L}_{p, \As}^{\imp}(\underline{\Pi})_{|\kappa_\Uu = a} = \Omega_p(\Pi[a]) \cdot \mathcal{L}_{p, \As}^{\imp}(\Pi[a]).\]
  Combining this result with \cref{thm:families} (which gives the existence of a family through any given $\p_1$-ordinary $\Pi$) gives \cref{thmintro:twovar}.

\section{Evaluating local zeta-integrals}

 We now compute the classical specialisations of the measure defined above and prove the interpolation property of Theorem \ref{thmintro:cycloL}.

 \subsection{Non-archimedean case: setup}
  Let $K$ be a nonarchimedean local field, of residue characteristic $\ell$, and $L$ a finite \'etale $K$-algebra of degree 2. Thus $L$ is either $K \times K$, or a quadratic field extension. (In our applications, $K$ will be $\Q_\ell$, and $L$ will be $F \otimes_{\Q} \Q_\ell$ for a real quadratic field $F$.) We write $q_K$, resp.~$q_L$, for the orders of the residue fields of $K$ and $L$ (understood as $q_K = q_L$ if $L = K \times K$).

  We fix a choice of nontrivial character $\psi : L \to \C^\times$; we can and do assume that $\psi$ is trivial on $K \subset L$, and that $\psi$ is unramified (i.e.~it is trivial on $\Oo_L$ but not on any larger fractional ideal).

  We set $H = \GL_2(K)$, and $G = \GL_2(L)$, so $H \subset G$. Write $B_H, N_H, Z_H$ etc for the usual subgroups of $H$, and similarly for $G$.

  \subsubsection{Godement--Siegel sections}

   For $\chi$ a smooth character of $K^\times$, and $\Phi \in \Ss(K^2, \C)$ and $h \in H$, recall
   \[
    f^{\Phi}(h; \chi, s) = |\det h|_K^s \int_{K^\times} \Phi( (0, a) h) \chi(a)|a|_K^{2s}\, \mathrm{d}^\times a.
   \]
   This defines a meromorphic section of the family of principal series representations $I(|\cdot|_K^{s-1/2}, |\cdot|_K^{1/2 - s} \chi^{-1})$ of $H$.

  \subsubsection{Whittaker models, new vector theory}

   Let $\pi$ be a generic irreducible smooth representation of $G$, which we suppose to be unitary and tempered. Then we have a Whittaker model $\Ww(\pi, \psi)$ (we shall omit $\psi$ henceforth).

   \begin{thm}[Casselman]
    There exists an ideal $\mathfrak{c} = \mathfrak{c}(\pi)$ of $\Oo_L$ with the following property: for $\mathfrak{a}$ any ideal of $\Oo_L$, the invariants of $\pi$ under the subgroup $K_1(\mathfrak{a}) \coloneqq \{ g \in \GL_2(\Oo_L) : g = \stbt{*}{*}{0}{1} \bmod \mathfrak{a}\}$ are nonzero iff $\mathfrak{c} \mid \mathfrak{a}$. The invariants at level $K_1(\mathfrak{c})$ are one-dimensional, and in the Whittaker model $\Ww(\pi, \psi)$, there is a unique generator $W^{\new}$ of this subspace satisfying $W^{\new}(1) = 1$.
   \end{thm}

   The values of $W^{\new}$ along the maximal torus can be extracted from the $\GL_2$ standard zeta integral:

   \begin{prop} \
    \begin{enumerate}
     \item Suppose $L$ is a field and $\varpi_L$ a uniformizer of $L$. Then we have
     \[ \sum_{n \ge 0} W^{\new}\left( \stbt{\varpi_L^n}{0}{0}{1} \right) q_L^{(1/2 - s)n} = L^{\mathrm{std}}(\pi, s) \]
     where $L^{\mathrm{std}}(-, s)$ denotes the standard $L$-factor.
     \item Suppose $L = K \times K$ (so $\pi = \pi_1 \times \pi_2$ for representations $\pi_i$ of $\GL_2(K)$), and let $\varpi_1 = (\varpi_K, 1)$ and $\varpi_2 = (1, \varpi_K)$ be uniformizers of the two maximal ideals of $\Oo_L$. Then
     \[ \sum_{n_1, n_2 \ge 0} W^{\new}\left( \stbt{\varpi_1^{n_1} \varpi_2^{n_2}}{0}{0}{1} \right) q_K^{(1/2 - s)(n_1 + n_2)} = L^{\mathrm{std}}(\pi_1, s)L^{\mathrm{std}}(\pi_2, s). \]
    \end{enumerate}
   \end{prop}

 \subsection{The Asai zeta integral}

  \subsubsection{Definitions}
   \label{sect:zetadef}

   For $W \in \Ww(\pi)$ and $\Phi \in \Ss(K^2, \C)$, we define
   \[ Z(W, \Phi; s) \coloneqq \int_{Z_H N_H \backslash H} W(h) f^{\Phi}(h; \chi, s)\, \mathrm{d}h\]
   where $\chi = \omega_{\pi}|_{K^\times}$. This is well-defined for $\re(s) \gg 0$ and has analytic continuation as a meromorphic function of $q_K^s$; as usual, we define $L_{\As}(\pi, s)$ as the unique $L$-factor which generates the fractional ideal of $\C[q_K^{\pm s}]$ given by the values of this integral as $W$ and $\Phi$ vary.

  \subsubsection{Torus integral}

   We define for $W \in \Ww(\pi)$ a ``partial zeta integral''
   \[ Z(W; s) = \int_{K^\times} W\left( \stbt x 0 0 1 \right) |x|_K^{s - 1}\, \mathrm{d}^\times x.\]
   Cf.~\cite[Definition 3.3.1]{myhida}. Again, the integral is convergent for $\re(s) \gg 0$ and has analytic continuation as a meromorphic function of $q_K^s$.

   \begin{prop}
    The function on $H$ given by $h \mapsto Z(h W; s)$ transforms as an element of the induced representation $I(|\cdot|^{1/2 - s}, |\cdot|^{s - 1/2} \chi)$, i.e.~the dual of the representation in which the Siegel section lies; and we have
    \[ \int_{B_H \backslash H}  Z(hW; s) f^{\Phi}(h; \chi, s)\, \mathrm{d}h = Z(W, \Phi, s).\]
   \end{prop}

   \begin{proof}
    This follows by splitting up the integral over $Z_H N_H \backslash H$ into an integral over $Z_H N_H \backslash B_H$ followed by an integral over $B_H \backslash H$.
   \end{proof}

   Note that it suffices to take the integral over $K_H = \GL_2(\Oo_K)$, since $B_H K_H = H$.

  \subsubsection{A special Schwartz function}

   Let $r \ge 1$, and consider the Schwartz function
   \[ \Phi_r = (q_K^2 - 1)q_K^{2r - 2} \cdot \ch( \p_K^r, 1 + \p_K^r) \]
   where $\mathfrak{p}_K$ is the maximal ideal of $\Oo_K$.
   From \cite[Lemma 3.2.5]{GSP4}, if $\chi$ is trivial on $1 + \p_K^r$, then $f^{\Phi_r}(-; \chi, s)$ vanishes outside the coset $B_H \backslash B_H K_0(\p_K^r)$ (the ``small cell'' modulo $\p^r$), and we have
   \[ f^{\Phi_r}(1; \chi, s) = (q_K + 1) q_K^{r - 1}, \]
   which is also the inverse of the volume of the small cell for the unramified Haar measure on $B_H \backslash H$. (On the other hand, if $\chi$ is nontrivial on $1 + \p^r$ then $f^{\Phi_r}(-; \chi, s)$ is identically zero.)

   Since the small cell modulo $\p^r$ is the orbit of the identity under the subgroup $\stbt{1}{0}{\p^r}{1 + \p^r}$, we have the following: if $W \in \Ww(\pi)$ is invariant under this subgroup, the integrand $h \mapsto Z(hW; s) f^{\Phi}(h; \chi, s)$ is equal to $f^{\Phi_r}(1) Z(W, s)$ on the small cell in $(B_H \cap K_H) \backslash K_H$, and vanishes everywhere else on $K_H$. So for any $W$ invariant under this subgroup we have
   \[ Z(W, \Phi_r, s) = Z(W; s). \]

 \subsection{Imprimitive Asai $L$-factors}\label{sec:imprimitiveasai}

  \begin{defi}
   We define the \emph{imprimitive Asai $L$-factor} by
   \[ L^{\imp}_{\As}(\pi, s) =
    \begin{cases}
     Z(W^{\new}, s) & \text{if $\pi$ is ramified}\\
     L(\chi, 2s) \cdot Z(W^{\new}, s) & \text{if $\pi$ is unramified}.
    \end{cases}
   \]
  \end{defi}

  (Note that ``unramified'' here means ``has invariants under $\GL_2(\Oo_L)$'', even when $L$ is a ramified extension.)

  \begin{prop}\label{prop:zetaintegralimprimitive}
   Let $r_0$ be such that $\varpi_K^{r_0} = \mathfrak{c} \cap \Oo_K$, where $\mathfrak{c} \trianglelefteqslant \Oo_L$ is the conductor of $\pi$. Then we have
   \[ L^{\imp}_{\As}(\pi, s) = Z(W^{\new}, \Phi_{r_0}, s). \]
   If $r_0 \ge 1$ then we have in fact $L^{\imp}_{\As}(\pi, s) = Z(W^{\new}, \Phi_r, s)$ for any $r \ge r_0$.
  \end{prop}

  \subsubsection{Unramified case}

   Suppose $\pi$ has trivial conductor (but we allow $L$ to be a ramified extension). Then the new vector is simply the normalised spherical Whittaker function $W^\sph$; and the function $h \mapsto Z(h W^\sph, s)$ is invariant under $K_H$, so we want to take $r = 0$. In this case $f^{\Phi_r}(1; \chi, s) = L(\chi, 2s)$, so by a well-known explicit computation we find
   \[ L_{\As}^{\imp}(\pi, s) = L(\chi, 2s) Z(W^\sph, s) = L_{\As}(\pi, s). \]

   \begin{rmk}
    This computation is well-known if $L = K \times K$ or $L$ is an unramified quadratic extension, in which case $L_{\As}(\pi, s)$ is a degree 4 $L$-factor. However the computation is also valid for $L$ a ramified field extension, in which case $L_{\As}$ is a degree 3 $L$-factor (and coincidentally agrees with the symmetric square $L$-factor of $\pi$).
   \end{rmk}

  \subsubsection{Explicit formulae}

   We give explicit formulae for the imprimitive Asai $L$-factor defined above.

   \begin{prop}
    Suppose $L$ is a field extension (ramified or unramified), and $\pi$ has non-trivial conductor. Let $C$ be the multiset of nonzero complex numbers such that
    \[ L^{\mathrm{std}}(\pi, s) = \prod_{\alpha \in C} (1 - \alpha q_L^{-s})^{-1}.\]

    Then
    \[ L^{\imp}_{\As}(\pi, s) = \prod_{\alpha \in C} (1 - \alpha^{e_{L/K}} q_K^{-s})^{-1} \]
    where $e_{L/K}$ is the ramification index (1 or 2).\qed
   \end{prop}

   \begin{prop}
    Suppose $L = K \times K$ (so again $q_K = q_L \eqqcolon q$), and $\pi = \pi_1 \times \pi_2$ is arbitrary (ramified or unramified). For $i = 1, 2$ let $C_i$ be the multiset of complex numbers such that $L^{\mathrm{std}}(\pi_i, s) = \prod_{\alpha_i \in C_i} (1 - \alpha_i q^{-s})^{-1}$. Then we have
    \[ L^{\imp}_{\As}(\pi, s) = \prod_{\alpha_1 \in C_1} \prod_{\alpha_2 \in C_2} (1 - \alpha_1 \alpha_2 q^{-s})^{-1}.\]
   \end{prop}

  \subsubsection{Relation to the primitive Asai $L$-factor}

   In all cases there exists an $L$-factor $P(s)$ such that
   \[ L_{\As}(\pi, s) = P(s) \cdot L^{\imp}_{\As}(\pi, s).\]
   We have seen that if $\pi$ is unramified, then $P(s) = 1$. One checks similarly that if $L = K \times K$, then we have $P(s) = 1$ in any of the following cases:
   \begin{itemize}
    \item at least one of $\pi_1$, $\pi_2$ is unramified;
    \item one of the $\pi_i$ is supercuspidal, but the other is not.
   \end{itemize}

 \subsection{Zeta integrals with $U'$-eigenvectors}

  We now consider the local zeta integral for $\ell = p$, taking $K = \Qp$ and $L = F \otimes \Qp \cong \Qp \times \Qp$. We take the additive character of $L$ to be the restriction to $F \otimes \Qp$ of the global character $\psi_F : \A_F \to \C^\times$ of \cref{not:addchar}. Since we want to consider twists by ramified characters, we shall consider the more general integral
  \[ Z(W, \Phi; \hat\nu_p^{-1}, s) \coloneqq \int_{(Z_H N_H \backslash H)(\Qp)} W(h) \hat\nu_p^{-1}(\det h) f^{\Phi}(h; \chi \hat\nu_p^{-2}, s)\,\mathrm{d}h \]
  where $\nu$ is a Dirichlet character.

  We have $\pi = \pi_1 \times \pi_2$ for representations $\pi_i$ of $\GL_2(\Qp)$. If $\psi_p$ is the standard additive character of $\Qp$, we can identify $\Ww(\pi_1, \psi_p) \otimes \Ww(\pi_2, \psi_p)$ with $\Ww(\pi, \psi_{F, p})$ via mapping $W_1 \otimes W_2$ to the function $g \mapsto W_1(\stbt{1/\sqrt{D}}{}{}{1} g_1) W_2(\stbt{-1/\sqrt{D}}{}{}{1} g_2)$ (so normalised new vectors go to normalised new vectors). A straightforward substitution gives
  \[
   Z(W_1 \otimes W_2, \Phi; \hat\nu_p^{-1}, s) = \hat\nu_p^{-1}(-\sqrt{D}) \cdot \int_{(Z_H N_H \backslash H)(\Qp)} W_1(\stbt{-1}{}{}{1}g) W_2(g) \hat\nu_p^{-1}(\det g) f^{\Phi}(g; \chi \hat\nu_p^{-2}, s)\, \mathrm{d}g,
  \]
  and the integral on the right-hand side is a Rankin--Selberg zeta integral. We are interested in the following situation:

  \begin{itemize}
   \item $\pi_1$ and $\pi_2$ are unramified;
   \item $\nu$ is a Dirichlet character of $p$-power conductor (possibly trivial), so $\hat\nu_p^{-1}(-\sqrt{D}) = \nu(-\sqrt{D})$;
   \item $\Phi = \Phi_{p, \nu^{-1}, \nu}$ as defined above;
   \item $W_2$ is the normalised spherical vector;
   \item $W_1$ is the $U_{\p_1}'$-eigenvector $W_{\p_1, \alpha_1}[n]$ defined above, for some $n \gg 0$.
  \end{itemize}

  \begin{prop}

   With the above choices, we have
   \[ Z(W_1 \otimes W_2, \Phi; \hat\nu_p^{-1}, s) = \tfrac{p}{p+1} \nu(-\sqrt{D})\Ee_p(\As \Pi, \nu^{-1}, s) L_{\As}(\Pi_p \times \hat\nu_p^{-1}, s).\]
  \end{prop}

  This is a rather standard computation (compare e.g.~\cite{chenhsieh}) so we shall not give the proof in detail here. For a more general formula implying the result, see Proposition A.7 of \cite{LZ-improved}, which gives a general formula for $\GL_2 \times \GL_2$ zeta-integrals with $U_p'$-eigenvectors in the first $\GL_2$ factor (and arbitrary choices of $W_2$ and $\Phi$) as a product of a $\GL_2$ $\gamma$-factor (which gives the factor $\Ee_p$) and an integral over $\Qp^\times$ involving the product of $W_1$ and a Whittaker function coming from $\Phi$. For the choice of $W_2$ and $\Phi$ made above, the integrand in this torus integral is the characteristic function of $\Zp^\times$, so the integral is just 1 (see \S A.4 of op.cit.). Compare e.g.~\cite[Prop 8.15]{padicLfnct}, which is an analogous formula in the rather more complicated setting of $\operatorname{GSp}_4 \times \GL_2$ instead of $\GL_2 \times \GL_2$.

 \subsection{Archimedean case}

  It remains to compute the Archimedean zeta integral, for which we use the formula of \cref{def:whittakerinf} for $W_\infty$. Since $\hat\nu_\infty$ is trivial on the support of $W_\infty$, we need to compute
  \[
   Z(W_{\infty}, \Phi_\infty^{(k_1 - k_2)}, s)\coloneqq\int_{Z(\R)N(\R)\backslash \GL_2(\R)}W_\infty(h)f^{\Phi_\infty}(h;1,s)\, \mathrm{d}h.
  \]
  As shown in \cite[p.85]{jacquet}, since the integrand is invariant under $\R^{\times}\times {\rm SO}_2(\R)$, we can write this as
  \[
   f^{\Phi_\infty^{(k_1 - k_2)}}(1; 1, s) \cdot \int_{\R^\times_+} y^{(k_1 + k_2)/2} e^{-4\pi y / \sqrt{D}} y^{s - 1}\, \mathrm{d}^\times y.
  \]
  We have
  \[
   \int_{\R^\times_+} y^{(k_1 + k_2)/2} e^{-4\pi y / \sqrt{D}} y^{s - 1}\, \mathrm{d}^\times y =
   \tfrac{1}{2} \cdot \left( \tfrac{\sqrt{D}}{2}\right)^{s - 1 + \tfrac{k_1 + k_2}{2}}\cdot \Gamma_{\C}\left(s - 1 + \tfrac{k_1 + k_2}{2}\right), \]
  where $\Gamma_\C(s)\coloneqq2(2\pi)^{-s}\Gamma(s)$, and similarly
  \[ f^{\Phi_\infty^{(k_1 - k_2)}}(1; 1, s) =
   i^{k_1 - k_2} 2^{s-\tfrac{k_1 - k_2}{2}}\Gamma_{\C}\left(s + \tfrac{k_1 - k_2}{2}\right),
  \]
  so we conclude that
  \[ Z(W_{\infty} ,{\Phi_\infty}^{(k_1 - k_2)},s) =
   i^{k_1 - k_2} 2^{-k_1} (\sqrt{D})^{s - 1 + (k_1 + k_2)/2} \cdot \Gamma_{\C}(s - 1 + \tfrac{k_1 + k_2}{2}) \Gamma_{\C}(s + \tfrac{k_1 - k_2}{2}).\]

 \subsection{The global zeta integral} We can finally put together the zeta integrals of the above sections to prove our main result.

  \begin{thm}
   Let $\Pi$ be an automorphic representation as above, with central character $\chi_{\Pi}$. Let $\mathcal{L}_{p, \As}(\Pi)$ be the one-variable measure defined in Definition \ref{defi:onevariable}. Then for any locally-algebraic character $s+\rho$ of $\Zp^\times$, with $s$ an integer such that $\tfrac{1-k_1 + k_2}{2} \le s \le \tfrac{k_1 - k_2}{2}$ and $\rho$ of finite order, we have
   \[
    \mathcal{L}_{p, \As}(\Pi)(s+\rho) = \Ee_p(\As(\Pi), \rho^{-1}, s) \cdot \frac{\Gamma(s + \tfrac{k_1+k_2-2}{2}) \Gamma(s + \tfrac{k_1 - k_2}{2})}{2^{(k_1-2)} i^{(1-k_2)} (-2\pi i)^{(2s + k_1 - 1)}} \cdot \frac{L_{\As}\left(\Pi, \rho^{-1}, s\right)}{\Omega_\infty(\Pi)}.
   \]
   The same holds with $\mathcal{L}_{p, \As}^{\imp}(\Pi)$ and $L_{\As}^{\imp}(\Pi, \rho^{-1}, s)$ in place of the primitive $L$-functions.
  \end{thm}

  \begin{proof}
   We first consider the primitive $L$-function. By definition and the discussion in \cref{sect:onevarpadic}, we have  for $n \gg 0$
   \begin{multline*}
    \mathcal{L}_{p, \As}(\Pi)(s+\rho) = \tfrac{(p+1)}{p} \rho^{-1}(-\sqrt{D}) \sqrt{D}^{1 - (k_1 + k_2)/2 -s}  (-1)^s
    \times \\ \frac{L_{\As}(\Pi \times \rho^{-1}, s)}{\Omega_\infty(\Pi)} \cdot \frac{Z(W_{p, \alpha_1}[n], \Phi_{p, \rho^{-1}, \rho}; \rho_p, s)}{L_{\As}(\Pi_p \times \rho_p^{-1}, s)} \cdot Z(W_\infty; \Phi_\infty^{(k_1 - k_2)}, s).
   \end{multline*}
   Substituting the formulae we have just obtained for the local integrals at $p$ and at $\infty$ gives the claimed formula.

   In the case of the imprimitive $L$-function, we use \cref{cor:zetaint}, which gives a formula for $\mathcal{L}_{p, \As}^{\imp}(\Pi)(s+\rho)$ which is the product of the formula above for $\mathcal{L}_{p, \As}(\Pi)(s+\rho)$ and the extra terms $\prod_{\ell \ne p} C_\ell(W_\ell, \Phi_\ell, \ell^{-s} \rho(\ell)^{-1})$, where the $W_\ell$ and $\Phi_\ell$ are as specified in \cref{sect:imprimtest}. So it suffices to show that for these choices, for each $\ell \ne p$ we have
   \[ C_\ell(W_\ell, \Phi_\ell, \ell^{-s} \rho(\ell)^{-1}) = \frac{L^{\imp}_{\As}(\Pi_\ell \times \hat\rho_\ell^{-1}, s)}{L_{\As}(\Pi_\ell \times \hat\rho_\ell^{-1}, s)}.\]
   Since $\hat\rho_\ell$ is unramified, it suffices to prove the relation for $\hat\rho_\ell = 1$ and all complex $s$, which is exactly \cref{prop:zetaintegralimprimitive}.
  \end{proof}

 \subsection*{Acknowledgements}

  This paper was written while all three authors were in residence at the Mathematical Sciences Research Institute in Berkeley, California, for the program ``Algebraic Cycles, $L$-values and Euler Systems'' during the spring of 2023 (supported by NSF grant no.~DMS-1928930). It is a pleasure to thank MSRI and its staff for providing such a splendid working environment.

  %\bibliographystyle{../amsalphaurl}
  %\bibliography{padicAsai}
  \providecommand{\bysame}{\leavevmode\hbox to3em{\hrulefill}\thinspace}
  \renewcommand{\MR}[1]{%
   MR \href{http://www.ams.org/mathscinet-getitem?mr=#1}{#1}.
  }
  \newcommand{\articlehref}[2]{\href{#1}{#2}}

\end{document}